\documentclass[10pt,a4paper]{amsart}

\usepackage{array}
\usepackage{lipsum}
\usepackage{comment}
\usepackage{color}

\usepackage[utf8]{inputenc}
\usepackage[T1]{fontenc}
\usepackage{textcomp}
\usepackage{amsmath,amssymb,amsopn,amsthm}
\usepackage{lmodern}
\usepackage[a4paper]{geometry}
\usepackage{graphicx}
\usepackage{xcolor}
\usepackage{microtype}

\usepackage{calrsfs}

\usepackage{dsfont}

\usepackage{amsmath,amsfonts,amssymb,amsthm,mathrsfs,unicode}
\usepackage{calligra}

\usepackage{pdfpages}

\usepackage{tikz}
\usetikzlibrary{decorations.pathreplacing, arrows.meta}

\usepackage{hyperref}

\usepackage{graphicx}
\usepackage{mathtools}

\hypersetup{pdfstartview=XYZ}

\newcommand{\union}[2]{\text{$\underset{#1}{\bigcup} #2$}}

\newcommand{\intervEnt}[2]{\text{$\llbracket #1,#2\rrbracket$}}

\newcommand{\somme}[4]{\text{$\overset{#2}{\underset{#4=#1}{\sum}} #3$}}

\newcommand{\sommeInd}[2]{\text{${\underset{#1}{\sum}} #2$}}

\newcommand{\prodend}[2]{\text{${\underset{#1}{\prod}} #2$}}

\newcommand{\integrale}[4]{\text{$\int_{#1}^{#2} #3 d#4$}}

\newcommand{\integraleMes}[3]{\text{$\int_{#1} #2 d#3$}}

\newcommand{\norme}[1]{\text{$\left\Vert #1\right\Vert$}}

\newcommand{\abso}[1]{\text{$\left\vert #1\right\vert$}}

\newcommand{\prive}[2]{\text{$#1 \setminus #2$}}

\newcommand{\reciproque}[1]{\text{${#1}^{-1}$}}

\newcommand{\intInd}[3]{\text{$\int_{#1} #2 d#3$}}

\newcommand{\restreint}[2]{\text{$#1_{\left \vert #2 \right.}$}}

\newtheorem{definition}{Definition}[section]

\newtheorem{theorem}[definition]{Theorem}

\newtheorem{lemme}[definition]{Lemma}

\newtheorem{coro}[definition]{Corollary}

\newtheorem{proposition}[definition]{Proposition}

\newtheorem{exemple}[definition]{Example}

\newtheorem{remark}[definition]{Remark}

\newcounter{numberSection}[section]

\newcounter{numberSubSection}[subsection]

\addtocounter{numberSection}{1}
\addtocounter{numberSubSection}{1}

\begin{document}
\title{Keplerian shear with Rajchman property}
\author{Arthur Boos \and Benoit Saussol}
\address{Aix-Marseille Université, Institut de Mathématiques de Marseille, CNRS, Luminy case 907, 13288 Marseille}

\begin{abstract}
The Keplerian shear was introduced within the context of measure preserving dynamical systems by Damien Thomine \cite{DaTho}, as a version of mixing for non ergodic systems. In this study we provide a characterization of the Keplerian shear using Rajchman measure, for some flows on tori bundles. 
Our work applies to dynamical systems with singularities or with non-absolutely continuous measures. We relate the speed of decay of conditional correlations with the Rajchman order of the measures.
Some of these results are extended to the case of compact Lie group
bundles.
\end{abstract}
\date{\today}
\maketitle

\tableofcontents

\section{Introduction}

In a dynamical system, the mixing property reveals that trajectories are intermingled and asymptotically distributed somewhat homogeneously. A paradigmatic example is provided by Arnold's cat map on the 2-torus
$\mathbb{T}^2$, endowed with the Lebesgue measure,
\[
T=\left(\begin{matrix}
2 & 1\\
1 & 1
\end{matrix}\right).\]

From a probabilistic perspective, mixing signifies the asymptotic independence of events; long-term evolution forgets the initial conditions. This property, along with its quantitative counterparts (e.g., decay of correlations), forms the foundation for various probabilistic limit theorems within the context of deterministic dynamics, such as the Central Limit Theorem and Borel-Cantelli lemmas (see \cite{Chern}).

However, there are many systems in which certain quantities are conserved during evolution. This occurs in integrable systems of Hamiltonian dynamics, particularly in some geodesic flows and rational billiards (see \cite{DaTho} and references therein). This phenomenon can also manifest in biological systems, where the preservation of a character results in different evolutionary paths. The presence of these invariants prevents the system from being ergodic and, consequently, from exhibiting mixing. The transvection
\[T := \left(\begin{matrix}
    1 & 0\\
    1 & 1
\end{matrix}\right),\]
acting on the torus $\mathbb{T}^2$,
is a paradigmatic example.
This map acts on the second coordinate as a rotation on the circle. The evolution in each ergodic component (each individual system) is notably simple and predictable. Nevertheless, Kesten demonstrated that, with some randomness introduced into the rotation angle, trajectories distribute quite homogeneously in the long term. Specifically, the discrepancies, suitably normalized, converge to a Cauchy distribution \cite{kesten} (see also \cite{DolgoFay} for multidimensional generalization).

In celestial mechanics, planetary rings (the motion of each dust particle) may be modeled by the flow
\[\begin{array}{lll}
g_t : &[a,b]\times \mathbb{T} &\to  [a,b]\times \mathbb{T}\\
&(r,\theta)&\mapsto \left(r,\theta+tr^{-\frac{3}{2}}\right).\end{array}\]
While the distance to the center is preserved, and each trajectory is essentially a rotation, the fact that angular velocities vary allows for the aggregation of materials, potentially explaining the formation of larger bodies.
Recently Damien Thomine \cite{DaTho} introduced the notion of Keplerian shear, formalizing the fact that in non ergodic systems, trajectories may distribute homogeneously and independently of their past, provided we ignore the invariants.

The aim of our work is to pursue the study of non-ergodic dynamical systems which have the property of Keplerian shear.

The probabilistic dynamical system $(X,T,\mu)$ has Keplerian shear if for all $f\in L^2(\mu)$ we have the weak convergence
\begin{equation}\label{1}
 f\circ T^n \rightharpoonup E_\mu(f|\mathcal{I}),
 \end{equation}
$\mathcal{I}$ being the $\sigma$-algebra of invariant by the transformation $T$.
The lack of ergodicity is indeed adding randomness to the dynamics (the choice of the ergodic component), and even if the dynamics on the fiber is not mixing the system can globally appear mixing conditionally to the fibers. 

We now present the organization of the article.
In the section 2, we define Keplerian shear, relate it to the mixing property and the {Rajchman} property of a measure, the main notion in this article. We also make a description of the  dynamical systems that we consider in this work, which includes locally action-angle dynamics.

In the section 3, we show the main result in the discrete case and we investigate the speed of shearing, using 
 anisotropic {Sobolev} spaces.

Section 4 addresses  in the continuous case the same questions of section 3.

Section 5 provides applications of Keplerian shear to dynamical Borel-Cantelli lemmas and Diophantine approximation.

To finish, section 6 generalizes the work to the case where the phase space is a  connected-compact Hausdorff {Lie} group.
\\

\noindent 
{\it Acknowledgements.}
We would like to thank the anonymous referee, as well as Sébastien Gouezel and Thierry De La Rue for their time and valuable inputs to our 
manuscript. Their constructive suggestions and valuable comments are critical to improving the overall quality of the manuscript.

\section{Settings}
\subsection{Mixing property and Keplerian shear}

The mixing property is well known and is satisfied in many situations. However it makes sense only for ergodic systems. From this point of view, Keplerian shear may be a right compromise.

We will note in the article $(\Omega,\mathcal{T},\mu,(g_t)_{t\in\mathbb{R}})$ (resp.  $(\Omega,\mathcal{T},\mu,T)$) a continuous dynamical (resp. discrete) system.
A mixing system is an asymptotically independent system in the following sense:

\begin{definition}[Mixing system]
We say that $(\Omega,\mathcal{T},\mu,(g_t)_{t\in\mathbb{R}})$ (resp. $(\Omega,\mathcal{T},\mu,T)$) is mixing if for all $A,B\in\mathcal{T},$

\begin{equation}\mu(A\cap g_t^{-1}(B))-\mu(A)\mu(B)\xrightarrow[t\to+\infty]{}0.\end{equation}

\begin{equation}\left(\text{resp. }  \mu(A\cap \left(T^n\right)^{-1}(B))-\mu(A)\mu(B)\xrightarrow[n\to+\infty]{}0.\right)\end{equation}

\end{definition}

We recall that mixing systems are ergodic. Consequently the following systems are not mixing, but we will show that they exhibit Keplerian shear.

\begin{exemple}[Non ergodic systems]
We endow these systems with the Lebesgue measure on $\mathbb{T}^2$.

\begin{enumerate}
\item $\begin{array}{lclc}T : &\mathbb{T}^2&\to&\mathbb{T}^2\\
&(x,y)&\mapsto &(x,y+x)\end{array}$
\item $\begin{array}{lclc}g_t : &\mathbb{T}^2&\to&\mathbb{T}^2\\
&(x,y)&\mapsto &\left(x,y+t\cos\left(2\pi\left(x-\frac{1}{2}\right)\right)\right).\end{array}$
\end{enumerate}
\end{exemple}

\begin{definition}[Invariant $\sigma-$algebra]
The invariant $\sigma-$algebra $\mathcal{I}$ is the $\sigma-$algebra of measurable sets invariant by continuous flow (resp. discrete):

For $(\Omega,\mathcal{T},\mu,(g_t)_{t\in\mathbb{R}})$ (resp. $(\Omega,\mathcal{T},\mu,T)$) a continuous dynamical system  (resp.discrete)

$\mathcal{I} := \left\{A\in\mathcal{T} : \forall t\in \mathbb{R},\mu(A\Delta g_t^{-1}(A))=0\right\} (\text{resp. } \left\{A\in\mathcal{T} :\mu(A\Delta T^{-1}(A))=0\right\})$. 

\end{definition}

The Keplerian shear is a notion of asymptotic independence conditionally to the $T-$invariant algebra.

\begin{definition}[Keplerian shear]

The dynamical system $(\Omega,\mathcal{T},\mu,(g_t)_{t\in\mathbb{R}})$ exhibits Keplerian shear if for all $f\in \mathbb{L}^2_\mu(\Omega)$:

\begin{equation}\label{eq:kepl_shear_def}f\circ g_t\underset{t\to+\infty}{\rightharpoonup}\mathbb{E}_\mu(f\vert\mathcal{I})\, \left(\text{resp.} f\circ T^n \underset{n\to+\infty}{\rightharpoonup}\mathbb{E}_\mu(f\vert\mathcal{I})\right).\end{equation}
\end{definition}


\begin{definition}[Conditional correlation]

We define the conditional correlation for $f_1,f_2\in \mathbb{L}^2_\mu(\Omega)$ by
\begin{equation}Cov_t(f_1,f_2\vert\mathcal{I}):=\mathbb{E}_\mu(\overline{f_1}\cdot (f_2\circ g_t)\vert\mathcal{I})-\overline{\mathbb{E}_\mu(f_1\vert \mathcal{I})} \cdot\mathbb{E}_\mu(f_2\vert \mathcal{I}) ;\end{equation}

\[\left(\text{resp. }Cov_n(f_1,f_2\vert\mathcal{I}):=\mathbb{E}_\mu(\overline{f_1}\cdot (f_2\circ T^n)\vert\mathcal{I})-\overline{\mathbb{E}_\mu(f_1\vert \mathcal{I})} \cdot\mathbb{E}_\mu(f_2\vert \mathcal{I})\right).\]

\end{definition}

\begin{proposition}[\cite{DaTho}]
The Keplerian shear property is equivalent to the convergence to $0$ of the expectation of the conditional correlation, in other words, for all $f_1,f_2\in \mathbb{L}^2_\mu(\Omega)$,
\begin{equation}
\mathbb{E}_\mu\left(Cov_t(f_1,f_2\vert \mathcal{I})\right)\xrightarrow[t\to+\infty]{}0 ;\left(\text{resp. } \mathbb{E}_\mu\left(Cov_n(f_1,f_2\vert \mathcal{I})\right)\xrightarrow[n\to+\infty]{}0 \right).
\label{correlation}
\end{equation}
\end{proposition}


In ergodic systems, Keplerian shear is equivalent to  mixing.
The mixing property is thus equivalent to the ergodicity and Keplerian shear.

\subsection{Rajchman measure}

The Fourier theory will be instrumental in our study of correlation decay. 
In particular, measures with the Riemann-Lebesgue property will play an important role in this work.

\begin{definition}[Rajchman measure]
In the continuous case, a measure $\nu$ on  $(\mathbb{R},\mathcal{B}(\mathbb{R}))$ is Rajchman if
\begin{equation}\label{eq:rajchman}\widehat{\nu}(t)\xrightarrow[t\to\pm\infty]{}0\,;\end{equation}

In the discrete case, a measure $\nu$ on $(\mathbb{T},\mathcal{B}(\mathbb{T}))$
is {Rajchman} if
\begin{equation}\label{eq:rajchman}\widehat{\nu}(n)\xrightarrow[n\to\pm\infty]{}0,\end{equation}
where
$\hat{\nu}(t)= \intInd{\mathbb{R}}{e^{2i\pi tx}}{\nu(x)}\,; \left(resp.  \hat{\nu}(n)= \intInd{\mathbb{T}}{e^{2i\pi nx}}{\nu(x)}\right)$.
\end{definition}

\subsubsection{Functional equivalences}

We recall an equivalent interpretation of the Rajchman property in terms of convergence in distribution or equidistribution.


\begin{proposition}
\label{equitore}
In the continuous case (resp.discrete), the measure $\nu$ is Rajchman if and only if for $x$ distibuted by $\nu$ and for all $a>0$,

\begin{equation} \frac{nx}{a} \mod 1 \xrightarrow[n\to+\infty]{\mathscr{L}}\mathcal{U}\left(\mathbb{T}\right);\end{equation}

\begin{equation}\left( resp.\,nx \mod 1\xrightarrow[n\to+\infty]{\mathscr{L}} \mathcal{U}\left(\mathbb{T}\right)\right).\end{equation}

\end{proposition}

\begin{proof}
We begin with the discrete case.
Let $\varphi\in\mathcal{C}^2(\mathbb{T})$. Its Fourier series decomposition gives for all $x\in \mathbb{T}$ and all $n\in\mathbb{Z}$,
\begin{equation}\label{eq:serie_fourier}\varphi(nx)=\sommeInd{k\in \mathbb{Z}}{\widehat{\varphi}(k)e^{2i\pi kn x}}.\end{equation}

We recall that \begin{equation}\label{eq:zero_fourier}
    \widehat{\varphi}(0)=\intInd{\mathbb{T}}{\varphi(x)}{\lambda(x)}
\end{equation}
and for $m\in\mathbb{Z},$ \begin{equation}\label{eq:carac}\widehat{\nu}(m)=\int_{\mathbb{T}}e^{2i\pi m x}d\nu(x).\end{equation}

By Lebesgue convergence theorem, \eqref{eq:serie_fourier},\eqref{eq:zero_fourier} and \eqref{eq:carac}, we get
\begin{equation*}\int_{\mathbb{T}}\varphi(nx)d\nu(x)=\intInd{\mathbb{T}}{\varphi(x)}{\lambda(x)}+\sommeInd{k\in \mathbb{Z}^*}{\widehat{\varphi}(k)}\widehat{\nu}(kn).\end{equation*}

Again, by Lebesgue convergence theorem,
\begin{equation*}
    \sommeInd{k\in \mathbb{Z}^*}{\widehat{\varphi}(k)}\widehat{\nu}(kn)\xrightarrow[n\to+\infty]{}0.
\end{equation*}
By density this convergence extends to $\varphi\in\mathcal{C}^0(\mathbb{T})$ and concludes the proof.

For the reciprocal implication, we take $\varphi : x\mapsto e^{2i\pi x}$, by \eqref{eq:serie_fourier} we get \begin{equation}\label{eq:carac_int}\widehat{\nu}(n)=\intInd{\mathbb{T}}{\varphi(nx)}{\nu(x)}\end{equation} and \begin{equation}\label{eq:vanish_zero}
    \widehat{\varphi}(0)=0.
\end{equation}
By hypothesis, \eqref{eq:zero_fourier}, \eqref{eq:carac_int} and \eqref{eq:vanish_zero},
\begin{equation}
    \widehat{\nu}(n)\xrightarrow[n\to+\infty]{}0.
\end{equation}

Now, we treat the continuous case.
Let's start with the direct implication.

Let $a>0$ and let $\mu$ be the push forward of $\nu$ by $x\mapsto\frac{x}{a}\mod 1$.
We have for any $n\in\mathbb{Z}$
\begin{equation*}\widehat{\mu}(n)=\intInd{\mathbb{T}}{e^{2i\pi nx}}{\mu(x)}=\intInd{\mathbb{R}}{e^{2i\pi\frac{n}{a}x}}{\nu(x)}=\widehat{\nu}\left(\frac{n}{a}\right)\xrightarrow[n\to+\infty]{}0.\end{equation*}
Therefore, $\mu$ is Rajchman on $\mathbb{T}$ and we apply the result obtained in the discrete case.

Now, we prove the reciprocal implication.
Let $\varepsilon>0$.
The uniform continuity of the characteristic function gives that there exists $\delta>0$ such that for all $x,y\in\mathbb{R}$,
\begin{equation}\label{eq:uniform_continuity}\vert x-y\vert<\delta\Longrightarrow \vert \widehat{\nu}(x)-\widehat{\nu}(y)\vert<\frac{\varepsilon}{2} .\end{equation}
In addition, by the Rajchman hypothesis, there exists $n_0\in\mathbb{N}$ such that for all $n\in \mathbb{N}$,  \begin{equation}\label{eq:range_control}n\geq n_0\Longrightarrow \widehat{\nu}\left(n\frac{\delta}{2}\right)<\frac{\varepsilon}{2}.\end{equation}

Let $t\geq n_0\frac{\delta}{2}$.
Since $\frac{\delta}{2}\mathbb{N}\cap [t,t+\delta[\ne\emptyset$, there exists $ n\in \mathbb{N},\vert n\frac{\delta}{2}-t\vert<\delta$. Thus, by \eqref{eq:uniform_continuity} and \eqref{eq:range_control},
 \begin{equation*}\vert\widehat{\nu}(t)\vert\leq \left\vert \widehat{\nu}(t)-\widehat{\nu}\left(n\frac{\delta}{2}\right)\right\vert +\left\vert\widehat{\nu}\left(n\frac{\delta}{2}\right)\right\vert<\varepsilon.\end{equation*}

\end{proof}


\begin{lemme}[Weak$-*$ convergence and {Rajchman} property]
\label{convfaible}

Let $\mu$ be a Borel-probability measure on $\mathbb{R}$.

The measure $\mu$ is {Rajchman} if and only if \begin{equation}\label{eq:weak_topo_*}\left(x\mapsto e^{2i\pi tx}\right)\xrightarrow[t\to+\infty]{}0\end{equation} in the weak-$*$ topology on $\left(\mathbb{L}^1_\mu(\mathbb{R})\right)^* = \mathbb{L}^\infty_\mu(\mathbb{R})$.

\end{lemme}

\begin{proof}
The reciprocal implication is obtained immediately applying the duality property \eqref{eq:weak_topo_*} with $(x\in\mathbb{R}\mapsto 1)\in \mathbb{L}^1_\mu(\mathbb{R})$.

Now, we prove the direct implication.
Let $\varphi\in \mathcal{S}(\mathbb{R})$.
Writing $\varphi$ with its inverse {Fourier} transform, we get for all $x\in\mathbb{R}$,
\begin{equation}\label{eq:reverse_fourier}
    \varphi(x)=\intInd{\mathbb{R}}{e^{2i\pi xs}\widehat{\varphi}(s)}{\lambda(s)}.
\end{equation}
By Fubini theorem and \eqref{eq:reverse_fourier}, we get
\begin{equation*}
    \intInd{\mathbb{R}}{e^{2i\pi tx}\varphi(x)}{\mu(x)}=\intInd{\mathbb{R}}{\widehat{\varphi}(s)\widehat{\mu}(t+s)}{\lambda(s)}.
\end{equation*}
Since $\mu$ is Rajchman, by Lebesgue convergence theorem
\begin{align*}
\intInd{\mathbb{R}}{\widehat{\varphi}(s)\widehat{\mu}(t+s)}{\lambda(s)}\xrightarrow[t\to+\infty]{} 0.
\end{align*}
\end{proof}

\begin{lemme}[Weak$-*$ convergence and {Rajchman} property in discrete case]
\label{convfaible_discrete}

Let $\mu$ be a Borel-probability measure on $\mathbb{T}$.

The measure $\mu$ is {Rajchman} if and only if \begin{equation}\label{eq:weak_topo_*_discrete}\left(x\mapsto e^{2i\pi nx}\right)\xrightarrow[n\to+\infty]{}0\end{equation} in the weak-* topology on $\mathbb{L}^2_\mu(\mathbb{T})$.
\end{lemme}

\begin{proof}
The reciprocal implication is obtained immediately applying the duality property \eqref{eq:weak_topo_*_discrete} with $(x\in\mathbb{R}\mapsto 1)\in \mathbb{L}^2_\mu(\mathbb{T})$.

Now, we prove the direct implication.
Let $\varphi\in \mathcal{C}^\infty(\mathbb{T})$.
Expanding $\varphi$ in Fourier serie we get for all $x\in\mathbb{T}$,
\begin{equation}\label{eq:reverse_fourier_discrete}
    \varphi(x)=\sum_{k\in\mathbb{Z}}\widehat{\varphi}(k)e^{2i\pi k x}.
\end{equation}
By Fubini theorem and \eqref{eq:reverse_fourier}, we get
\begin{equation*}
    \intInd{\mathbb{T}}{e^{2i\pi nx}\varphi(x)}{\mu(x)}=\sum_{k\in\mathbb{Z}}\widehat{\varphi}(k)\widehat{\mu}(k+n)\xrightarrow[n\to+\infty]{} 0
\end{equation*}
by Lebesgue convergence theorem, using the fact that $\mu$ is Rajchman.
The conclusion follows by density.
\end{proof}
\subsubsection{Fourier-{Rajchman} decay}

The notion of Rajchman measure is only qualitative. To estimate the speed of convergence to $0$ of the expectation of conditional correlations~\eqref{correlation},  a quantitative version of the Fourier decay will be needed.


\begin{definition}[Fourier-{Rajchman} decay]

A Rajchman speed of a Rajchman measure on $\mathbb{T}\,(resp.\, \mathbb{R})$ is a value $r\geq 0$ such that there exists $C>0$ satisfying for all $n\in\mathbb{Z}$ (resp. $t\in\mathbb{R}$),

\begin{equation}\label{eq:fourier_decay_discrete}\abso{\widehat{\nu}(n)}\leq \frac{C}{\abso{n}^r};\end{equation}

\begin{equation}\label{eq:fourier_decay_continuous} \text{resp. }\abso{\widehat{\nu}(t)}\leq \frac{C}{\abso{t}^r}.\end{equation}

In other words, $\widehat{\nu}(n)=O_\infty(\abso{n}^{-r})$ (resp. $\widehat{\nu}(t)=O_\infty(\abso{t}^{-r})$).
\end{definition}

\begin{definition}[{Rajchman} order]

The Rajchman order of a Rajchman measure is the supremum of its Rajchman speeds.
We denote it $r(\mu)$.

\end{definition}
\begin{remark}
The Rajchman order is linked to the Fourier dimension $dimF(\mu)$ defined in \cite{BoSol} by the relation $dimF(\mu)=2r(\mu)$.
Additionnaly,  as mentioned there, when the {Rajchman} order $r$ verifies $r>\frac{1}{2}$, $\mu\ll\lambda$. 
Consequently a Rajchman measure can be singular continuous only if its Rajchman order $r$ satisfies $r\leq \frac{1}{2}$.
\end{remark}

\begin{definition}[Diophantine exponent]
We define the Diophantine exponent of $x\in \mathbb{R}$ by $Dio(x)=\inf A = \sup B$, where\footnote{With the convention $\inf\emptyset=+\infty$.}
\begin{equation}A=\left\{s>0 : \exists C>0,\forall (p,q)\in \mathbb{Z}\times \mathbb{N}^*,\abso{qx-p}\geq \frac{C}{q^s}\right\}\end{equation}
and \begin{equation}B=\left\{t>0 : \exists \, \text{infinitely many} \, (p,q)\in \mathbb{Z}\times \mathbb{N}^*,\abso{qx-p}< \frac{1}{q^t}\right\}.\end{equation}

\end{definition}

\begin{definition}[$s-$Diophantine number]\label{def:s_diophante}
Let $s>0.$

A number $x\in\mathbb{R}$ is called $s$-Diophantine if there exists $C>0$ such that for all $p\in\mathbb{Z}$, for all $q\in \mathbb{N}^*,$
\begin{equation}
\abso{qx-p}\geq \frac{C}{q^s} \label{diophante}.
\end{equation}
Let $\mathcal{D}io(s)$ be the set of $s$-Diophantine numbers.
\end{definition}

\begin{definition}[Liouville Number]
A number $x\in \mathbb{R}$ is {Liouville} if and only if for all $s>0$, $x$ is not $s$-Diophantine, that is, $Dio(x)=+\infty$.
In other words \[x\in \mathbb{L} := \mathbb{R}\setminus \left(\union{s>0}{\mathcal{D}io(s)}\right).\] 
\end{definition}

The Rajchman order of a measure controls the almost sure Diophantine exponent (See Proposition~\ref{kurzman})

\begin{proposition}
\label{diophraj}

If $r(\mu)\leq \frac{1}{2}$, then  $\mu-a.s\, \alpha \in [0,1[, Dio(\alpha)\leq \frac{1}{r(\mu)}-1$. In the absolutely continuous case, $Dio(\alpha)=1$ a.s.

\end{proposition}

\subsubsection{Radon-Nikodym-Lebesgue decomposition}

We discuss the simple relations that can be seen between the Rajchman property and the Radon-Nikodym-Lebesgue decomposition \begin{equation}\label{eq:RNL_decomposition}\nu = \nu_{ac} + \nu_{sc} + \nu_{d},\end{equation}
with $\nu_{ac}$ the absolutely continuous part, $\nu_{sc}$ the singular continuous part and $\nu_d$ the discrete part.
\begin{enumerate}
\item{\underline{Absolutely continuous part}.}
$\nu_{ac}$ is an absolutely continuous measure and theorefore, it is Rajchman. This is garanteed by Riemann-Lebesgue.
When we know the regularity of the density, we can estimate the Rajchman order. The work of Damien Thomine adresses the case of absolutely continuous measures (See Theorem~\ref{thomcis}).
\item{\underline{Singular continuous part}.}
Many singular continuous measures are Rajchman and other are not. The main discussion is about the Rajchman property of $\nu_{sc}$.
We will extend Damien Thomine results for continuous singular measures (See Theorem~\ref{mtcont}).
\item{\underline{Discrete part}.}
Discrete measure are never Rajchman because the non-convergence of complexe exponential to $0$.
\end{enumerate}

Consequently, $\nu$ is Rajchman if and only if $\nu_{sc}$ is {Rajchman} and $\nu_d=0$.


\subsubsection{Example of singular Rajchman measures}
\label{examRajch}

\begin{definition}[{Pisot} number]

A real $\theta$ is a Pisot number if and only if  $\theta >1$ is algebraic and every other roots $\theta_r$ of its minimal polynomial satisfy $\abso{\theta_r}<1$.

\end{definition}

We note that Lebesgue-almost all numbers are not Pisot.
Now, to highlight the Rajchman property for the singular measure, we will present examples of continuous singular measures having this property. 
\bigbreak
\begin{exemple}[Self-similar measure]
Consider for $\theta>2$, the distribution $\mu_\theta$ of $\sum_{n\in\mathbb{N}^*}\pm \theta^{-n}$, where signs are chosen with i.i.d probabilities $\frac{1}{2}$.
Note that $\mu_{\theta}$ has as characteristic function $\widehat{\mu_\theta} : k\in\mathbb{Z}\mapsto \prod_{n\in\mathbb{N}}{\cos\left(2\pi\theta^{-n}k\right)}$.


$\mu_{\theta}$ is Rajchman if and only if $\theta$  is not a Pisot number \cite{Kaha}.

Additionally, for {Lebesgue}-almost all real $\theta>2$, these measures are singular and the  {Rajchman} order $r\left(\mu_\theta\right)>0$.
\end{exemple}

\begin{exemple}[{Rajchman} measure with {Liouville} set as support $(Dio=+\infty)$]
\label{cantraj}

{Christian Bluhm} \cite{CBluhm} has built a Rajchman measure $\mu_{\infty}$ supported on Liouville numbers.

According to the Proposition \ref{diophraj}, its order is $r\left(\mu_{\infty}\right)=0$.
\end{exemple}


\subsection{Tori bundle}~\label{toribundle}

We study a class of dynamical systems introduced in \cite{DaTho}, tori bundles. These systems have local action-angle coordinates.

\begin{definition}[Tori bundle]
Let $(M,\mathcal{A})$ be a $n\in\mathbb{N}^*$  dimensional $\mathcal{C}^1$ Lindelöf\footnote{Any open cover has a countable subcover.} manifold.

Let $d\in \mathbb{N}^*$, $(\Omega,\mu)$ a Borel space and $\pi$ a continuous projection of $\Omega$ to $M$.


$\Omega$ is a $(n,d)$ dimensional tori bundle if
\begin{enumerate}
\item{locally, we have for charts $U$ of $\mathcal{A}$ an homeomorphism : $\psi_U : \pi^{-1}(U)\to U\times \mathbb{T}^d$}
\item{for all $U$ in $\mathcal{A},\pi_1\circ \psi_U=\pi_{\vert U}$, where $\pi_1$ is the projection on the first coordinate.}

\end{enumerate}
\end{definition}

Bundles of the form $M\times \mathbb{T}^d$ constitute already interesting examples.


\begin{definition}[Compatible flow]
Let $(\Omega, \mu, (g_t)_{t\in\mathbb{R}})$ be a measure preserving dynamical system.

The flow $(g_t)_{t\in \mathbb{R}}$ is a compatible flow with $(\Omega,\mu)$ as a tori bundle if for all charts $U\in \mathcal{A}$ there exists $v_U \in \left(\mathbb{R}^d\right)^U$ measurable such that for $t\in\mathbb{R}, \psi_U\circ g_t\circ \psi_U^{-1}(x,y)=(x,y+tv_U(x))$.

We note $g_t^U : \psi_U\circ g_t\circ \psi_U^{-1}$.
\end{definition}

\begin{definition}[Compatibles measure]
A measure $\mu$ is compatible if for all charts $U\in \mathcal{A},\left(\mu_{\vert\pi^{-1}(U)}\right)_{\psi_U}=\left(\mu_{\pi}\right)_{\vert U}\otimes \lambda$.
\end{definition}

\section{Main results in discrete dynamical systems}

In this case, we use the {Rajchman} property with the help of the corresponding {Fourier} series.
We will highlight the necessary path to show the presence of Keplerian shear in this case.
We now state the main result in the discrete case.

\begin{theorem}\label{main:discrete}
The discrete dynamical system $(\Omega, \mu, T)$ with $\Omega$ a tori bundle exhibits Keplerian shear if and only if for all $0\neq\xi\in\mathbb{Z}^d$, and for all charts $U\in\mathcal{A},m^\mathbb{T}_{\xi,U}$ is {Rajchman} where \begin{equation}m^\mathbb{T}_{\xi,U}=\left(\left(\left(\mu_{\pi}\right)_{\vert U}\right)_{\langle\xi\vert v_U(\cdot)\rangle-\lfloor \langle\xi\vert v_U(\cdot)\rangle\rfloor}\right)_{\vert\mathbb{T}\setminus \{0\}}.\end{equation}
\end{theorem}

\begin{proof}
Let us start with the direct implication.
Let $\xi\in\mathbb{Z}^d$, $\xi\neq0$, and $U\in\mathcal{A}$.
Take \[f_1 : x\in \pi^{-1}(U)\mapsto \mathds{1}_{\left(\langle\xi\vert v_U(\cdot)\rangle-\lfloor \langle\xi\vert v_U(\cdot)\rangle\rfloor\right)^{-1}({\mathbb{T}\setminus\{0\}})}(\pi(x))e^{2i\pi\langle\xi\vert \pi_2\circ \psi_U(x)\rangle}.\]
Take $f_2 =f_1$. For all $n\in\mathbb{N}$, \[
\intInd{\mathbb{T}}{e^{2i\pi n z}}{m^\mathbb{T}_{\xi,U}(z)}=
\intInd{\pi^{-1}(U)}{\overline{f_1}\cdot f_2\circ T^n}{\mu(x)}
\xrightarrow[n\to+\infty]{} \mathbb{E}_\mu\left(\mathbb{E}_\mu(\overline{f_1}\mathds{1}_{\pi^{-1}(U)}|\mathcal{I})\mathbb{E}_\mu(f_2|\mathcal{I})\right)
\]
by Keplerian shear.
Using the explicit form of $f_2$ and {Birkhoff} Theorem we find that  $\mathbb{E}_\mu(f_2\vert \mathcal{I})=0$, therefore
\[\intInd{\mathbb{T}}{e^{2i\pi n z}}{m^\mathbb{T}_{\xi,U}(z)}\xrightarrow[n\to+\infty]{}0.\]
Hence $m^\mathbb{T}_{\xi,U}$ is {Rajchman}.

Let us treat the reciprocal implication.
Using a partition of unity the problem reduces, for
one chart $U\in\mathcal{A}$ and two functions $f_1$ and $f_2$ in $\mathbb{L}^2(U\times\mathbb{T}^d,{\mu_\pi}_{\vert U}\otimes\lambda)$, to the convergence of the covariance 
\[
C_n:=\intInd{U\times \mathbb{T}^d}{\overline{f_1}(x,y)f_2(x,y+nv_U(x))}{(\mu_\pi)_{\vert U}\otimes\lambda(x,y)}.
\]
By Lemma \ref{convfaible_discrete}, and the Fourier decomposition on $\mathbb{T}^d$ of $f_1(x,\cdot)$ and $f_2(x,\cdot)$ for all charts $U\in\mathcal{A}$, all $x\in U$, we get
\begin{equation}\label{lasomme}
C_n=\sum_{\xi\in\mathbb{Z}^d} I_n(\xi), \text{ where } \intInd{U}{\overline{\widehat{f_1(x,\cdot)}(\xi)}\widehat{f_2(x,\cdot)}(\xi)e^{2in\pi\langle \xi\vert v_U(x)\rangle}}{(\mu_\pi)_{\vert U}(x)}.
\end{equation}

For $\xi\in\mathbb{Z}^d$, writing $a^{(1)}_\xi(x)=\widehat{f_1(x,\cdot)}(\xi)$ and $a^{(2)}_\xi(x)=\widehat{f_2(x,\cdot)}(\xi)$, we have,
\begin{equation}\label{eq:develop_cond}
    I_n(\xi)=\intInd{U}{\mathbb{E}\left(\overline{a^{(1)}_\xi}a^{(2)}_\xi\vert\langle \xi\vert v_U(x)\rangle\right)e^{2in\pi\langle \xi\vert v_U(x)\rangle}}{(\mu_\pi)_{\vert U}(x)}.
\end{equation}
Let $g$ be the measurable function such that $\mathbb{E}\left(\overline{a^{(1)}_\xi}a^{(2)}_\xi\vert\langle \xi\vert v_U(x)\rangle\right)=g(\langle \xi\vert v_U(x)\rangle)$.
By transfert theorem, for $\xi\ne 0$,
\begin{equation}\label{eq:domine_int}
    \intInd{U}{\mathbb{E}\left(\overline{a^{(1)}_\xi}a^{(2)}_\xi\vert\langle \xi\vert v_U(x)\rangle\right)e^{2in\pi\langle \xi\vert v_U(x)\rangle}\mathds{1}_{\langle\xi\vert v_U(x)\rangle\ne 0}}{(\mu_\pi)_{\vert U}(x)}=\intInd{\mathbb{T}}{g(z)e^{2i\pi n z}}{m^{\mathbb{T}}_{\xi,U}(z)},
\end{equation}
which converges to $0$ by hypothesis.

For $\xi\in\mathbb{Z}^d$, for all $n\in\mathbb{N}$ we have
\begin{equation}\label{eq:domine_sum}
    \abso{\intInd{U}{\overline{a^{(1)}_\xi(x)}a^{(2)}_\xi(x)e^{2in\pi\langle \xi\vert v_U(x)\rangle}}{(\mu_\pi)_{\vert U}(x)}}\leq \norme{a^{(1)}_\xi}_{\mathbb{L}^2}\norme{a^{(2)}_\xi}_{\mathbb{L}^2}
\end{equation}
and
\begin{equation}
    \sum_{\xi\in\mathbb{Z}^d}\norme{a^{(1)}_\xi}_{\mathbb{L}^2}\norme{a^{(2)}_\xi}_{\mathbb{L}^2}\leq \norme{f_1}_{\mathbb{L}^2}\norme{f_2}_{\mathbb{L}^2}.
\end{equation}
Therefore, using  \eqref{eq:develop_cond} and \eqref{eq:domine_int}, we can apply the dominated convergence theorem in \eqref{lasomme}, which proves that $C_n$ converges.
\end{proof}

\subsection{Previous results}
We recall some results obtained by Thomine on this topic in the discrete case for Lebesgue measure.

\begin{definition}[Anisotropic Sobolev spaces \cite{DaTho} Section 2.5.1]

Let $s>0$ and $h : (x,y)\in\mathbb{R}^2\mapsto \begin{cases}
\left(1+\frac{x^2}{y^2}\right)^{\frac{1}{2}}\, if \, y\ne 0\\
1\,else\end{cases}$.
We define an anisotropic {Sobolev} space of order $s>0$ like this :

\[H^{s,0}(\mathbb{T}^2) :=\left\{f\in \mathbb{L}^2(\mathbb{T}^2) : \sum_{\xi\in\mathbb{Z}^d}\abso{\widehat{f}(\xi)}^2h^{2s}(\xi)\in\mathbb{R}\right\}.\]
\end{definition}

Provided the observables have some regularity, one can get an upper estimate of the decay of conditional correlations for the transvection, as expressed below.
\begin{proposition}[\cite{DaTho} Proposition 2.10]
Considering the dynamical system $(\mathbb{T}^2,\lambda\otimes\lambda,T)$ with $T : (x,y)\in\mathbb{T}^2\mapsto(x,y+x)$, we have that: 
for $s>0$, for all $f_1,f_2\in H^{s,0}(\mathbb{T}^2),n\in\mathbb{N}^*$
\begin{equation} \abso{\mathbb{E}_{\lambda\otimes\lambda}(Cov_n(f_1,f_2\vert\mathcal{I}))}\leq \frac{4^s}{n^{2s}}\norme{f_1}_{H^{s,0}(\mathbb{T}^2)}\norme{f_2}_{H^{s,0}(\mathbb{T}^2)}.\end{equation}
\end{proposition}


\subsection{Generalizations to singular measures}
Here we present a version of the previous result for non-absolutely continuous measures.

\begin{definition}[Sobolev space on the torus]
We will note for $s>0,d\in\mathbb{N}^*$ the Sobolev space \[H^{s}(\mathbb{T}^d) := \left\{f\in \mathbb{C}^{\mathbb{T}^d} : \sum_{k\in\mathbb{Z}^d}\abso{\widehat{f}(k)}^2\left(1+\norme{k}_2^2\right)^{s}\in\mathbb{R}\right\}.\]
\end{definition}

\begin{definition}
We will note for $s>0,d\in\mathbb{N}^*$ and $f\in H^s(\mathbb{T}^d),$ \begin{equation}\label{eq:def_sup_sobolev}C_f(s) := \sup\left\{\abso{\widehat{f}(\xi)}\left(1+\norme{\xi}_2^2\right)^{\frac{s}{2}} : \xi\in \mathbb{Z}^d\right\}.\end{equation}
\end{definition}

Again with some regularity on the observables, and assuming in addition a positive Rajchman order on the distribution of the fibers, one can bound the decay of conditional correlations.

\begin{proposition} \label{rajchdisc}
Let $\mu$ be a measure on $\mathbb{T}$ of Rajchman order equal to $r>0$.
Considering the dynamical system $(\mathbb{T}^2,\mu\otimes\lambda,T)$ with $T : (x,y)\in\mathbb{T}^2\mapsto(x,y+x)$,
we have that:

for all $s>2$, there exists $\gamma_s\in \left[\min\left\{s-1,r\right\},r\right]$ such that for all $f_1,f_2\in H^s(\mathbb{T}^2)$, for all $\gamma<\gamma_s$, there exists $C>0$ satisfying for all $n\in \mathbb{N}^*$, 
\begin{equation}\abso{\mathbb{E}_{\mu\otimes\lambda}(Cov_n(f_1,f_2\vert\mathcal{I}))}\leq \frac{C}{n^\gamma},\end{equation}
where $\gamma_s$ is optimal in the sense that
\begin{equation}
    \gamma_s = \sup\{\alpha >0 : \exists C>0, \forall f_1,f_2\in H^s(\mathbb{T}^2), \forall n\ ,\abso{\mathbb{E}_{\mu\otimes\lambda} (Cov_n(f_1,f_2\vert\mathcal{I}))}\leq \frac{C}{n^\alpha}\}.
\end{equation}

\end{proposition}

\begin{proof}
Let $s>2$, $n\in \mathbb{N}^*$ and $f_1,f_2\in H^s(\mathbb{T}^2).$
We recall that for $i\in\{1,2\}$, using the definition \eqref{eq:def_sup_sobolev}, for $q_1,q_2\in\mathbb{Z}$
\begin{equation}\label{eq:cont_sobolev}
    \abso{\widehat{f_i}(q_1,q_2)(1+q_1^2+q_2^2)^{\frac{s}{2}}}\leq C_{f_i}(s).
\end{equation}

For $\xi_1,\xi_2\in\mathbb{Z}$, let \[g_{(\xi_1,\xi_2)} : x\in\mathbb{T}\mapsto e^{2i\pi\xi_1x}\left(\intInd{\mathbb{T}}{\overline{f_1}(x,y)e^{2i\pi\xi_2 y}}{\lambda(y)}\right).\]
By Fourier decomposition, for $\xi_1,\xi_2\in \mathbb{Z}^d$ and for all $x\in\mathbb{T}$,
\begin{equation}\label{eq:Fourier_partial}
    g_{(\xi_1,\xi_2)}(x)=\sum_{k\in\mathbb{Z}}\widehat{f_1}(k,\xi_2)e^{2i\pi (\xi_1-k)x}.
\end{equation}

By \eqref{eq:Fourier_partial},
\begin{equation}\label{eq:int_Fourier_relation}\intInd{\mathbb{T}^2}{\overline{f_1}\left(f_2\circ T^n\right)}{\left(\mu\otimes\lambda\right)}=\sum_{\xi\in\mathbb{Z}^2}\widehat{f_2}(\xi)\intInd{\mathbb{T}}{e^{2i\pi n\xi_2x}g_{(\xi_1,\xi_2)}(x)}{\mu(x)}\end{equation}
and for all $k,\xi_1,\xi_2\in\mathbb{Z}$,
\begin{equation}\label{eq:shift_Fourier}\widehat{g_{(\xi_1,\xi_2)}}(k)=\widehat{f_1}(\xi_1-k,\xi_2).\end{equation}

By \eqref{eq:cont_sobolev} and \eqref{eq:shift_Fourier}, \begin{equation}\label{eq:sup_Fourier}\abso{\widehat{g_{(\xi_1,\xi_2)}}(k)}(1+k^2)^{\frac{s}{2}}\leq C_{f_1}\frac{(1+k^2)^{\frac{s}{2}}}{(1+\xi_2^2+(\xi_1-k)^2)^{\frac{s}{2}}}.\end{equation}
Moreover, by \eqref{eq:int_Fourier_relation} and \eqref{eq:shift_Fourier}, for $n\in\mathbb{N}$,  \[\mathbb{E}_{\mu\otimes\lambda}(Cov_n(f_1,f_2\vert \mathcal{I}))=\sum_{\xi\in\mathbb{Z}\times\mathbb{Z}^*}\widehat{f_2}(\xi)\intInd{\mathbb{T}}{e^{2i\pi n\xi_2x}g_{(\xi_1,\xi_2)}(x)}{\mu(x)}.\]
and
\begin{equation}
    \label{eq:correlation}
    \mathbb{E}_{\mu\otimes\lambda}(Cov_n(f_1,f_2\vert \mathcal{I}))=\sum_{(k,\xi)\in\mathbb{Z}\times \mathbb{Z}\times\mathbb{Z}^*}\widehat{f_2}(\xi)\widehat{f_1}(k,\xi_2)\widehat{\mu}(\xi_1-k+n\xi_2).
\end{equation}
We estimate the sum above, considering three different cases. 
In the first case, we suppose that $k,\xi_1,\xi_2\in\mathbb{Z}$ and $n\in\mathbb{N}$ are such that \begin{equation}\label{eq:assumption_1.1}\abso{\xi_1-k+n\xi_2}\leq \frac{n\abso{\xi_2}}{4}\end{equation}
and
\begin{equation}\label{eq:assumption_1.2}
\abso{n\xi_2-k}\leq \frac{n\abso{\xi_2}}{2}.
\end{equation}
Therefore, \begin{equation}\label{eq:ineq_int_1}\abso{k}\geq \frac{n\abso{\xi_2}}{2},\end{equation} \begin{equation}\label{eq:ineq_int_2}\abso{\xi_1}\geq \frac{3}{4}n\abso{\xi_2}\end{equation} and \begin{equation}\label{eq:ineq_int_3}\abso{\widehat{\mu}(\xi_1-k+n\xi_2)}\leq 1.\end{equation}
Assuming \eqref{eq:assumption_1.1} and \eqref{eq:assumption_1.2}, by \eqref{eq:sup_Fourier}, \eqref{eq:ineq_int_1}, \eqref{eq:ineq_int_2} and \eqref{eq:ineq_int_3}, \begin{equation}\label{eq:first_control}\sum_{\substack{(k,\xi_1,\xi_2)\in \mathbb{Z}\times\mathbb{Z}\times\mathbb{Z}^*\\\abso{\xi_1-k+n\xi_2}\leq \frac{n\abso{\xi_2}}{4}\\\abso{n\xi_2-k}\leq \frac{n\abso{\xi_2}}{2}}}\abso{\widehat{f_2}(\xi_1,\xi_2)}\abso{\widehat{f_1}(k,\xi_2)}\abso{\widehat{\mu}(\xi_1-k+n\xi_2)}\leq \frac{2^{6(s+1)}\zeta(2(s-1))C_{f_1}(s)C_{f_2}(s)}{3^sn^{2(s-1)}}.\end{equation}

In the second case, we assume \eqref{eq:assumption_1.1} and \begin{equation}\label{eq:assumption_2.2}\abso{n\xi_2-k}> \frac{n\abso{\xi_2}}{2}.\end{equation}
Thus, we get again \eqref{eq:ineq_int_3} and \begin{equation}\label{eq:ineq_int_4}\abso{\xi_1}>\frac{n\abso{\xi_2}}{4}.\end{equation}
By \eqref{eq:ineq_int_3} and \eqref{eq:cont_sobolev},  $\abso{\widehat{f_1}(k,\xi_2)}\abso{\widehat{\mu}(\xi_1-k+n\xi_2)}\leq \frac{C_{f_1}(s)}{(1+k^2+\xi_2^2)^{\frac{s}{2}}}$
Moreover, by \eqref{eq:cont_sobolev} and \eqref{eq:ineq_int_4}, \begin{equation}\label{eq:ineq_int_5}\abso{\widehat{f_2}(\xi_1,\xi_2)}\leq \frac{4^sC_{f_2}(s)}{\abso{\xi_2}^{s}n^{s}}.\end{equation}
By \eqref{eq:ineq_int_5} and serie-integral criterius with $s>2$, there exists $C_s>0$ such that,  \[\sum_{\substack{(k,\xi_1,\xi_2)\in \mathbb{Z}\times\mathbb{Z}\times\mathbb{Z}^*\\\abso{\xi_1-k+n\xi_2}\leq \frac{n\abso{\xi_2}}{4}\\\abso{n\xi_2-k}> \frac{n\abso{\xi_2}}{2}}}\frac{C_{f_1}(s)\abso{\widehat{f_2}(\xi_1,\xi_2)}}{(1+k^2+\xi_2^2)^{\frac{s}{2}}}\leq \frac{C_s4^sC_{f_1}(s)C_{f_2}(s)}{n^{\left(s-1\right)}}.\]

Finally, in the third case we assume \begin{equation}\label{eq:assumption_3.1}\abso{\xi_1-k+n\xi_2}> \frac{n\abso{\xi_2}}{4}.\end{equation}
$\mu$ is a Rajchman measure and $r>0$ its order, by he definition \eqref{eq:fourier_decay_discrete}, there exists $C_\mu>0$ such that for $q\in\mathbb{Z}$,
\begin{equation}\label{eq:rajchman_order_constant}
    \abso{\widehat{\mu}(q)}\leq\frac{C_\mu}{\abso{q}^r}.
\end{equation}
Applying \eqref{eq:rajchman_order_constant} to $q=\xi_1-k+n\xi_2$ and by assumption \eqref{eq:assumption_3.1}, \begin{equation}\label{eq:mu_control}\abso{\widehat{\mu}(\xi_1-k+n\xi_2)}\leq \frac{4^rC_\mu}{(n\abso{\xi_2})^{r}}.\end{equation}
By \eqref{eq:mu_control}, \eqref{eq:cont_sobolev} and serie-integral criterius, for $s>2$ there exists a constant $C_s^'$ such that \begin{equation}\sum_{\substack{(k,\xi_1,\xi_2)\in \mathbb{Z}\times\mathbb{Z}\times\mathbb{Z}^*\\\abso{\xi_1-k+n\xi_2}> \frac{n\abso{\xi_2}}{4}}}\frac{C_{f_1}(s)\abso{\widehat{f_2}(\xi_1,\xi_2)}\abso{\widehat{\mu}(\xi_1-k+n\xi_2)}}{(1+k^2+\xi_2^2)^{\frac{s}{2}}}\leq \frac{4^r\left(1+\zeta\left(\frac{s}{2}\right)\right)C_\mu C_s^'C_{f_1}(s)C_{f_2}(s)}{n^{r}}.\end{equation}
Let pose  \[M(r,s)=\sup\left\{4^r\left(1+\zeta\left(\frac{s}{2}\right)\right)C_\mu C_s^',4^sC_s,2^{6(s+1)}\zeta(2(s-1))\right\}.\]
We get 
\begin{equation}\label{eq:estimation}\abso{\mathbb{E}_{\mu\otimes\lambda}\left(Cov_n(f_1,f_2\vert \mathcal{I})\right)}\leq M(r,s)\left(\frac{C_{f_1}C_{f_2}}{n^{\min\left\{s-1,r\right\}}}\right).\end{equation}

We now prove the optimality of the exponent.
The observable $f$ defined by $f(x,y)=e^{2i\pi y}$ belongs to $H^s(\mathbb{T}^2)$.
Moreover,
\begin{equation}
    \intInd{\mathbb{T}^2}{\overline{f}(x,y)f(T^n(x,y))}{\mu\otimes\lambda(x,y)}=\widehat{\mu}(n).
\end{equation}
Thus, for $\gamma<\gamma_s$, there exists $C>0$ such that for all $n\in\mathbb{N}^*$,
\begin{equation}
    \abso{\widehat{\mu}(n)}\leq \frac{C}{n^{\gamma}}.
\end{equation}
By definition of the Rajchman order $r$, $r\geq \gamma$. Therefore $r\geq \gamma_s$.
\end{proof}

\begin{remark}
The convergence speed is bounded by the order of {Rajchman} of the measure $\mu$; After some threshold, increasing the regularity will not accelerate the decay of correlation. 
On the other hand, this is consistent with the result obtained with the {Lebesgue} measure because its {Rajchman} order is $+\infty$, which removes the maximum imposed by $r$ for the speed of convergence and allows the observation of ever greater speeds of convergence when the regularity of the functions used increases.
\end{remark}

\subsection{Invariant functions and $\sigma-$algebra}

We can identify the $\sigma$-algebra of invariant sets for the discrete flow using {Fourier} series and orthogonality.

\begin{theorem}
Under the setting of Theorem~\ref{main:discrete}, a function $f\in \mathbb{L}^2_\mu(\Omega)$ is invariant according to $T$ if and only if for all $U\in\mathcal{A},$ there exists $(a_k)_{k\in\mathbb{Z}^d}\in \mathbb{L}^2_{\left(\mu_{\pi}\right)_{\vert U}}(U)^{\mathbb{Z}^d}$ such that $\left(\mu_{\pi}\right)_{\vert U}\otimes\lambda-a.e \,(x,y)\in U\times\mathbb{T}^d,$
\[\restreint{f}{\pi^{-1}(U)}(\psi_U^{-1}(x,y))=\sum_{k\in\mathbb{Z}^d, \langle k\vert v_U(x)\rangle\in\mathbb{Z}}a_k(x)e^{2i\pi\langle k\vert y\rangle}.\]
\end{theorem}

\begin{proof}
Let $f\in \mathbb{L}^2_\mu(\Omega)$ and $U\in\mathcal{A}$.

Let us prove the direct assertion.
Suppose that $f\circ T=f\,\mu-a.s$.
Let $(x,y)\in U\times \mathbb{T}^d$.
So  \[f(T(\psi_U^{-1}(x,y)))=f(\psi_U^{-1}(x,y)).\]
Namely $f(\psi_U^{-1}(x,y+v_U(x)))=f(\psi_U^{-1}(x,y))$.
By {Fourier} series decomposition: \begin{equation}\label{eq:fourier_inv}\sum_{k\in\mathbb{Z}^d}\widehat{f\circ \psi_U(x,\cdot)}(k)e^{2i\pi\langle k\vert y+v_U(x)\rangle}=\sum_{k\in\mathbb{Z}^d}\widehat{f\circ \psi_U(x,\cdot)}(k)e^{2i\pi\langle k\vert y\rangle}.\end{equation}
By uniqueness of the coefficients of a {Fourier} series in \eqref{eq:fourier_inv}, for $\mu-a.s\,x\in U$, for all $k\in\mathbb{Z}^d$,
\begin{equation}\label{eq:zero_inv}\widehat{f\circ \psi_U(x,\cdot)}(k)\left(e^{2i\pi \langle k\vert v_U(x)\rangle}-1\right)=0.\end{equation}
Thus $\langle k\vert v_U(x)\rangle\in\mathbb{Z}$ or $\widehat{f\circ \psi_U(x,\cdot)}(k)=0$.

Let's prove the reciprocal.
Suppose that $f$ satisfies for all charts $U\in \mathcal{A}$, 
$\left(\mu_{\pi}\right)_{\vert U}\otimes\lambda-a.s \,(x,y)\in U\times\mathbb{T}^d$
\begin{equation}\label{formeinv}
f(\psi_U^{-1}(x,y))=\sum_{k\in\mathbb{Z}^d,\langle k\vert v_U(x)\rangle\in\mathbb{Z}}a_k(x)e^{2i\pi\langle k\vert y\rangle}.
\end{equation}
Let $z\in \Omega$. There exists $U\in\mathcal{A}$ such that $z\in \pi^{-1}(U)$.
We have
 \[f(T(z))=f(T(\psi_U^{-1}(\psi_U(z))))=f(\psi_U^{-1}(\pi(z),\pi_2\circ \psi_U(z)+v_U(\pi(z)))).\]
Applying \eqref{formeinv} with the above computation, 
\[\begin{array}{ll}
f(T(z))&=
\displaystyle\sum_{k\in\mathbb{Z}^d,\langle k\vert v_U(x)\rangle\in\mathbb{Z}}a_k(\pi(z))e^{2i\pi\langle k\vert \pi_2\circ\psi_U(z)+v_U(\pi(z))\rangle}\\
&=\displaystyle\sum_{k\in\mathbb{Z}^d,\langle k\vert v_U(x)\rangle\in\mathbb{Z}}a_k( \pi(z))e^{2i\pi\langle k\vert \pi_2\circ\psi_U(z)\rangle}\\
&=f(z).
\end{array}
\]
\end{proof}


\section{Main result in continuous dynamical systems}\label{sec:main_cont}

As mentioned in the introduction, Damien Thomine gave many results which guarantee Keplerian shear in~\cite{DaTho}, in particular the following one.

\begin{theorem}[Result in the regular case, Theorem~3.3 in \cite{DaTho}]
\bigbreak
\label{thomcis}
Let $(\Omega, \mu, (g_t)_{t\in\mathbb{R}})$ be a compatible flow with a compatible measure on a tori (affine) bundle as in Subsection~\ref{toribundle}.

Assume that:
\begin{enumerate}
\item{$\mu \ll \lambda$},
\item{All velocity vectors $v_U$ are $\mathcal{C}^1$},
\item{$\forall U\in \mathcal{A},\mu\left(\union{\xi\in\mathbb{Z}^d\setminus\{0\}}{\{x\in U : d\langle\xi\vert v_U(x)\rangle=0\}}\right)=0$}.
\end{enumerate}
Then the system exhibits Keplerian shear. 
Moreover $$\mathcal{I} =\left\{B\in \mathcal{B}(\Omega) : \exists A\in\mathcal{B}(M),\mu\left(B\Delta \pi^{-1}(A)\right)=0\right\}.$$
\bigbreak
\end{theorem}

\begin{figure}
\begin{tikzpicture}
\draw[gray, thick] (-3,-3) -- (3,-3);
\draw[gray, thick] (-3,-2) -- (3,-2);
\draw[gray, thick] (-3,-1) -- (3,-1);
\draw[gray, thick] (-3,0) -- (3,0);
\draw[gray, thick] (-3,1) -- (3,1);
\draw[gray, thick] (-3,2) -- (3,2);
\draw[gray, thick] (-3,3) -- (3,3);
\draw[gray, thick] (-3,-3) -- (-3,3);
\draw[gray, thick] (-2,-3) -- (-2,3);
\draw[gray, thick] (-1,-3) -- (-1,3);
\draw[gray, thick] (0,-3) -- (0,3);
\draw[gray, thick] (1,-3) -- (1,3);
\draw[gray, thick] (2,-3) -- (2,3);
\draw[gray, thick] (3,-3) -- (3,3);
\draw[->](0,0) -- (3,1) node at (1.5,0.75)[anchor=-180, text = black]{$\xi$};
\draw [black,very thick] (-1,2) .. controls (-3,1) and (-1,0) .. (0,-2) node at (-3.1,1.4) [anchor=-180, text = black]{$v_U(U)$};
\draw[green,thick](0,0) -- (-3,-1);
\draw[blue,thick](0,-1) -- (-1,2);
\draw[blue,thick](-0.85,-2.1) -- (-1.85,0.9);
\draw[red,thick](-1.395,-0.465)-- (-0.3,-0.1) node at (-1.8,0.25) [anchor=180, text = red]{$\langle\xi\vert v_U(U)\rangle$};
\draw (-7,0) ellipse (2 and 1) node[anchor = east, text = black]{M};
\filldraw[gray] (-6,0) circle (0.5) node[anchor=center, text = white]{U};
\draw [->] (-6,0.5) .. controls (-3,3) .. (-1,2) node at (-2.5,2.5) [anchor=south, text = black]{$v_U$};
\draw node at (-1.1,-0.7) [anchor=180, text = red]{$m_{\xi,U}$};

\end{tikzpicture}
\caption{Push forward of measure  ${((\mu)_{\pi})}_{\vert U}$ on $\mathbb{R}$ by $\langle \xi\vert v_U\rangle$}
\label{mesimage}
\end{figure}

The proof of this theorem  uses Fourier transform, and relies on methods of differential geometry; In particular, the normal form of submersions.
We can see that the push forward measures $m_{\xi,U} :=\left(\left(\mu_{\pi}\right)_{\vert U}\right)_{\langle\xi\vert v_U(\cdot)\rangle}$ (see Figure~\ref{mesimage}) are absolutely continuous and verify {Riemann-Lebesgue} lemma.
This last property, in other words, {Rajchman} property will allow lonely to get Keplerian shear, it is even an equivalence.
We have successfully abstract this property and obtained the Keplerian shear under the {Rajchman} property.
Using tools from measure theory (conditional expectation)  allow to some extent to get rid of the regularity of the velocity.
The next theorem is the main result in continuous case, and refers to the real Rajchman property.
\begin{theorem}[The main result]\label{mtcont}
Let $(\Omega, \mu, (g_t)_{t\in\mathbb{R}})$ be a compatible flow with a compatible measure on a tori bundle as in Subsection~\ref{toribundle}.

The dynamical system exhibit Keplerian shear if and only if for all charts $U\in\mathcal{A}$, for all non zero $\xi\in \mathbb{Z}^d$, the push forward measures $\left(m_{\xi,U}\right)_{\vert \mathbb{R}^*}$ are Rajchman.
\end{theorem}

\begin{proof}
Let us prove the direct implication.
Let $U\in \mathcal{A}$ and non zero $\xi\in \mathbb{Z}^d$.
Take \[f_1 : z\in \pi^{-1}(U)\mapsto \mathds{1}_{\langle\xi\vert v_U\rangle\ne 0}(\pi(z))e^{2i\pi \langle\xi\vert \pi_2\circ \psi_U(z)\rangle}.\]
We have thanks to the compatibility of $\mu$
\begin{equation*}
\int_\Omega\overline{f_1}(z)\cdot\left(f_1\circ g_t(z)\right)d\mu(z)=\int_{\mathbb{R}}\mathds{1}_{\mathbb{R}^*}(s)e^{2i\pi ts}dm_{\xi,U}(s)
=
\int_{\mathbb{R}}e^{2i\pi ts}d\left(m_{\xi,U}\right)_{\vert \mathbb{R}^*}(s)
\end{equation*}
by push forward theorem.
Moreover, with Keplerian shear,
\[\int_\Omega\overline{f_1}(z)\cdot\left(f_1\circ g_t(z)\right)d\mu(z)\xrightarrow[t\to+\infty]{}\int_{\Omega}\overline{\mathbb{E}_\mu(f_1\vert\mathcal{I})}\mathbb{E}_\mu(f_1\vert \mathcal{I})d\mu=0\]
because by ergodic {Birkhoff} theorem \begin{align}\label{birkth}\mathbb{E}_\mu(f_1\vert \mathcal{I})\left(\psi_U^{-1}(x,y)\right)&\overset{\left(\mu_\pi\right)_{\vert U}\otimes \lambda-a.s}{=}e^{2i\pi\langle\xi\vert y\rangle}\mathds{1}_{\langle\xi\vert v_U\rangle\ne 0}(x)\underset{T\to +\infty}{\lim}\frac{1}{T}\integrale{0}{T}{e^{2i\pi t\langle\xi\vert v_U(x)\rangle}}{t}=0.\end{align}
So  $\left(m_{\xi,U}\right)_{\vert\mathbb{R}^*}$ is {Rajchman}.

Let us prove the reciprocal implication.
Let $(U_i)_{i\in I}\in \mathcal{A}^I$ be a countable partition of~$M$ modulo $\mu_{\pi}$.
Let  \[Y :=\union{(j,\xi)\in I\times \prive{\mathbb{Z}^d}{\left\{0\right\}}}{\left\{\left(a\circ\pi\right)\cdot e^{2i\pi\langle\xi\vert \pi_2\circ\psi_{U_j}\rangle}\in L^2_\mu\left(\reciproque{\pi}(U_j)\right) : a\in \mathbb{L}^\infty_\mu(U_j)\right\}}.\]
Let $(i,j)\in I^2$,$\left(a_1,a_2\right)\in \mathbb{L}^\infty_\mu(U_i)\times \mathbb{L}^\infty_\mu(U_j)$ and $\left(\xi_1,\xi_2\right)\in \left(\mathbb{Z}^d\right)^2.$
Take \[f_l = (a_l\circ \pi)\cdot e^{2i\pi\langle\xi\vert \pi_2\circ\psi_{U_j}\rangle}\] for $l\in\{1,2\}$.
When $i\ne j$, the supports are disjoint, and in this case \[\mathbb{E}_\mu\left(Cov_t\left(f_1,f_2\vert \mathcal{I}\right)\right)=0.\]
The same happens when $\xi_1\ne \xi_2$ by periodicity of complex exponentials.
Finally, when $\xi_1=\xi_2=0$, $\mathbb{E}_\mu\left(Cov_t\left(f_1,f_2\vert \mathcal{I}\right)\right)$ is constant.

Suppose without loss of generality that $i=j$ and~$\xi_1=\xi_2=\xi\ne 0.$
Set $W^{\xi}_i=U_i\setminus{\left(\langle\xi\vert v_{U_i}\rangle=0\right)}$.
We have
\[
\intInd{\Omega}{\overline{f_1}\cdot(f_2\circ g_t)\cdot\mathds{1}_{\langle\xi\vert v_{U_i}\rangle\ne 0}\circ \pi}{\mu}=\intInd{\mathbb{R}}{f(s)e^{2i\pi ts}}{(m_{\xi,U_i})_{\vert \mathbb{R}^*}(s)}
\xrightarrow[t\to\pm\infty]{}0
\]
because $(m_{\xi,U_i})_{\vert \mathbb{R}^*}$ is {Rajchman} thus converges weakly-$*$ according to  Lemma~\ref{convfaible}.

Moreover \[\intInd{\Omega}{\overline{f_1}\cdot(f_2\circ g_t)\cdot\mathds{1}_{\langle\xi\vert v_{U_i}\rangle= 0}\circ \pi}{\mu}=\intInd{\Omega}{ \overline{f_1}\cdot f_2\cdot\mathds{1}_{\langle\xi\vert v_{U_i}\rangle= 0}\circ \pi}{\mu}.\]

By Proposition 2.4 in \cite{DaTho}, we have that \[f_2\cdot\mathds{1}_{\langle\xi\vert v_{U_i}\rangle= 0}\circ \pi=\mathbb{E}_\mu\left(f_2\vert \mathcal{I}\right).\]
We obtains by totality of $Y$ that for all $f_1,f_2\in \mathbb{L}^2_\mu(\Omega)$
\[\mathbb{E}\left(Cov_t\left(f_1,f_2\vert \mathcal{I}\right)\right)\xrightarrow[t\to\pm \infty]{}0.
\]
\end{proof}

\begin{remark}
We can also note that the result generalises to infinite dimensional tori bundles $\mathbb{T}^\mathbb{N}$ with product topology.
\end{remark}

\begin{remark}
When Keplerian shear property holds, the decomposition of {Radon-Nikodym-Lebesgue} of the image measures $m_{\xi,U}$ in Theorem \ref{mtcont} reads as $\left(m_{\xi,U}\right)_{ac}+\left(m_{\xi,U}\right)_{sc}+\left(m_{\xi,U}\right)_{d}$ with $\left(m_{\xi,U}\right)_{d}=\alpha \delta_0 $, and $\alpha\geq 0$.
Recall that the absolutely continuous part satisfies the {Rajchman} property.
The discrete part is concentrated on $0$ because otherwise, a non-trivial periodicity would break the Keplerian shear.
The behavior of the singular part $\nu$ is not obvious, since it may be {Rajchman} or not, see subsection~\ref{examRajch}.
\end{remark}

\begin{remark}
When all the $m_{\xi,U}$'s are of the form $\left(m_{\xi,U}\right)_{ac}+\alpha_\xi\delta_0$, Theorem~\ref{mtcont} immediately gives Keplerian shear.
\end{remark}

\subsection{Invariant functions and $\sigma-$algebra}

Under the conditions of Theorem \ref{thomcis}, in the regular case, the invariant $\sigma-$algebra is just $\pi^{-1}(\mathcal{B}(M))$ modulo zero measure set.
On the functional aspect, it amounts to say that a measurable function $f$ is invariant by the flow if and only if it is $\mu-$a.s independent on the second variable on every chart $U$.
However, we will first highlight the invariant functions and then the $\sigma$-algebra of invariant sets in order to compare with the regular case mentioned above.
In the following theorem, we will use {Fourier} series to identify a characterization of invariance by an orthogonality property.

\begin{proposition}[Orthogonality and {Fourier} characterization of invariant functions]
Under the conditions of Theorem~\ref{mtcont}, a function 
$f\in \mathbb{L}^2_\mu(\Omega)$ is invariant according to $(g_t)_{t\in\mathbb{R}}$ if and only if for all $U\in\mathcal{A}, \exists (a_k)_{k\in\mathbb{Z}^d}\in \mathbb{L}^2_{\left(\mu_{\pi}\right)_{\vert U}}(U)^{\mathbb{Z}^d},\left(\mu_{\pi}\right)_{\vert U}\otimes\lambda-a.e \,(x,y)\in U\times\mathbb{T}^d$,
\[\restreint{f}{\pi^{-1}(U)}(\psi_U^{-1}(x,y))=\sum_{k\in\mathbb{Z}^d\cap\left\{v_U(x)\right\}^\perp}a_k(x)e^{2i\pi\langle k\vert y\rangle}.
\]
In other words, for a chart $U\in \mathcal{A}$ and for $x\in U$ fixed, non zero {Fourier} coefficients have indices orthogonal with the vector $v_U(x)$.
\label{orthofour}
\end{proposition}

\begin{proof}
Let $f\in \mathbb{L}^2_\mu(\Omega)$.
Let us start with the direct implication.
Suppose that for all $t\in \mathbb{R}$, \[f\circ g_t=f\,\mu-a.e.\]
Let $U\in\mathcal{A}$, $t\in\mathbb{R}$ and $(x,y)\in U\times \mathbb{T}^d$.
Since \[f(g_t(\psi_U^{-1}(x,y)))=f(\psi_U^{-1}(x,y))\]
we have \[f(\psi_U^{-1}(x,y+tv_U(x)))=f(\psi_U^{-1}(x,y)).\]
By  {Fourier} serie decomposition, this gives \[\sum_{k\in\mathbb{Z}^d}\widehat{f\circ \psi_U(x,\cdot)}(k)e^{2i\pi\langle k\vert y+tv_U(x)\rangle}=\sum_{k\in\mathbb{Z}^d}\widehat{f\circ \psi_U(x,\cdot)}(k)e^{2i\pi\langle k\vert y\rangle}.\]
Hence, by uniqueness of {Fourier} serie coefficients, for all  $k\in\mathbb{Z}^d$,
\[\widehat{f\circ \psi_U(x,\cdot)}(k)\left(e^{2i\pi t\langle k\vert v_U(x)\rangle}-1\right)=0.\]
When $k$ is such that $\widehat{f\circ \psi_U(x,\cdot)}(k)\ne 0$ we get that for all  $t\in \mathbb{R}, t\langle k\vert v_U(x)\rangle\in \mathbb{Z}$. 
So  $\langle k\vert v_U(x)\rangle=0$, and then $k\in \left\{v_U(x)\right\}^\perp$.
Hence  \[\left(\mu_{\pi}\right)_{\vert U}\otimes\lambda-a.e \,(x,y)\in U\times\mathbb{T}^d,f(\psi_U^{-1}(x,y))=\sum_{k\in\mathbb{Z}^d\cap\left\{v_U(x)\right\}^\perp}\widehat{f\circ \psi_U(x,\cdot)}(k)e^{2i\pi\langle k\vert y\rangle}.\]

Let see the reciprocal implication. Let $t\in \mathbb{R}$.
Suppose that $f$ satisfies for all charts $U\in \mathcal{A}$ \[\left(\mu_{\pi}\right)_{\vert U}\otimes\lambda-a.e \,(x,y)\in U\times\mathbb{T}^d,f(\psi_U^{-1}(x,y))=\sum_{k\in\mathbb{Z}^d\cap\left\{v_U(x)\right\}^\perp}a_k(x)e^{2i\pi\langle k\vert y\rangle}.\]
Let $z\in \Omega$ and take $ U\in\mathcal{A}$ such that $z\in \pi^{-1}(U)$.
We have  
\[
f(g_t(z))=f(g_t(\psi_U^{-1}(\psi_U(z))))=f(\psi_U^{-1}(\pi(z),\pi_2\circ \psi_U(z)+tv_U(\pi(z)))).\]
Thus  
\[
\begin{split}
f(g_t(z))&=\sum_{k\in\mathbb{Z}^d\cap\left\{v_U(\pi(z))\right\}^\perp}a_k(\pi(z))e^{2i\pi\langle k\vert \pi_2\circ\psi_U(z)+tv_U(\pi(z))\rangle} \\
&=\sum_{k\in\mathbb{Z}^d\cap\left\{v_U(\pi(z))\right\}^\perp}a_k(\pi(z))e^{2i\pi\langle k\vert \pi_2\circ\psi_U(z)\rangle}.
\end{split}
\]
by orthogonality. Hence \[f(g_t(z))=f(\psi_U^{-1}(\psi_U(z)))=f(z).\]
\end{proof}

\begin{remark} As a byproduct, the proposition gives an explicit form to the conditional expectation, with respect to the invariant $\sigma$-algebra, of a function $f\in \mathbb{L}^2_\mu(\Omega)$ locally by the next formula for $U\in\mathcal{A}$
\[ \left(\mu_{\pi}\right)_{\vert U}\otimes \lambda-a.e\, (x,y)\in U\times \mathbb{T}^d,\left(\mathbb{E}_\mu(f\vert \mathcal{I})\circ \psi_U^{-1}\right)(x,y) =\sum_{k\in \{v_U(x)\}^\perp\cap\mathbb{Z}^d}\widehat{(f\circ \psi_U^{-1})(x,\cdot)}(k)e^{2i\pi\langle k\vert y\rangle}.\]
\end{remark}

\begin{proposition}
The invariant $\sigma$-algebra is
\[\mathcal{I} =\left\{\union{U\in\mathcal{A}}{B_U} \in\mathcal{B}(\Omega) : (B_U)_{U\in\mathcal{A}}\in \prodend{U\in\mathcal{A}}{\psi_U^{-1}\left(\mathcal{I}_U\right)}\right\},
\]
where for $U\in\mathcal{A}$, the local invariant $\sigma-$algebra on $U$ is
\begin{equation}
\label{invartrib}
\mathcal{I}_U :=\left\{C\in \mathcal{B}(U\times \mathbb{T}^d) : \mu_{{\pi}_{\vert U}}-a.e\, x\in U,\widehat{\mathds{1}_C(x,\cdot)}^{-1}(\mathbb{C}^*)\subset \{v_U(x)\}^\perp\right\}.
\end{equation}
\end{proposition}

\begin{proof}
We prove the direct inclusion.
Let $A\in\mathcal{I}$. Let $U\in\mathcal{A}$ and set  $C=\psi_U(A\cap\pi^{-1}(U))$.
By Proposition \ref{orthofour} \[\mathds{1}_C(x,y)=\sum_{k\in \{v_U(x)\}^\perp\cap\mathbb{Z}^d}a_k(x)e^{2i\pi\langle k\vert y\rangle}\,\left(\mu_{\pi}\right)_{\vert U}\otimes \lambda-a.s\]
which shows that $C\in \mathcal{I}_U$. Hence
$A\cap \pi^{-1}(U)\in \psi_U^{-1}(\mathcal{I}_U)$, and finally \[A=\union{U\in\mathcal{A}}{A\cap\pi^{-1}(U)}.\]

Now, we prove the reciprocal inclusion.

Let $A\in \mathcal{B}(\Omega)$ s.t $\exists (B_U)_{U\in \mathcal{A}}\in \prodend{U\in\mathcal{A}}{\psi_U^{-1}\left(\mathcal{I}_U\right)},A=\union{U\in\mathcal{A}}{B_U}$.
Then, for all $t\in\mathbb{R}$, \[g_t^{-1}(A)=\union{U\in\mathcal{A}}{g_t^{-1}(B_U)}.\]
Let $U\in\mathcal{A}$.
By bijectivity of $\psi_U$, $\psi_U(B_U)\in\mathcal{I}_U$.
Hence \[\left(\mu_{\pi}\right)_{\vert U}\otimes\lambda-a.e\,(x,y)\in U\times\mathbb{T}^d,\mathds{1}_{\psi_U(B_U)}(x,y)=\sum_{k\in \{v_U(x)\}^\perp\cap\mathbb{Z}^d}\widehat{\mathds{1}_{\psi_U(B_U)}(x,\cdot)}(k)e^{2i\pi\langle k\vert y\rangle},\]
which proves that $A\in\mathcal{I}$ by Proposition \ref{orthofour}.
\end{proof}

\begin{remark} We can see that local invariants $A\in \mathcal{I}_U$ satisfy

\begin{equation}
\label{invarloc}
\mathds{1}_A(x,y) =\sommeInd{k\in \mathbb{Z}^d\cap \{v_u(x)\}^\perp}{\widehat{\mathds{1}_A(x,\cdot)}(k)e^{2i\pi\langle\xi\vert y\rangle}}\, \left(\mu_{\pi}\right)_{\vert U}\otimes \lambda-a.s.
\end{equation}
\end{remark}
\begin{lemme}\label{invchart}
We have  $\pi^{-1}(\mathcal{B}(M))\subset\mathcal{I}$.
\end{lemme}

\begin{proof}
Let $A=\pi^{-1}(B)$ for some $B\in \mathcal{B}(M)$.

Let $U\in\mathcal{A}$ and $t\in\mathbb{R}$.
We have  \[\mu-a.e \,z\in \pi^{-1}(U),\pi(z)=\pi_1\circ \psi_U(z).\]
And  \[\mu-a.e \,z\in \pi^{-1}(U),(\psi_U\circ g_t )(z)=(\psi_U\circ g_t\circ \psi_U^{-1})(\pi_1\circ \psi_U(z),\pi_2\circ \psi_U(z)).\]
So  \[\mu-a.e\,z\in \pi^{-1}(U),\pi_1((\psi_U\circ g_t )(z))=\pi(z)\] because tori bundle property of $\Omega$.
Namely \[{\mu-a.e\,z\in \pi^{-1}(U),\pi(g_t(z))=\pi(z)}.\]

We have $A\cap \pi^{-1}(U)=\pi^{-1}(B\cap U)$,
hence \[\psi_U(A\cap\pi^{-1}(U))=(B\cap U)\times\mathbb{T}^d\in \mathcal{I}_U.\]
Therefore $A$ is $g_t$ invariant by Proposition \ref{orthofour}.
\end{proof}

Note that the preceding inclusion is generally strict as shown in the following example.

\begin{exemple}
Consider $g_t =Id_{\mathbb{T}^2}$, $\Omega=\mathbb{T}^2$, $M=\mathbb{T}$.

Clearly $\mathcal{I}=\mathcal{B}(\mathbb{T}^2)$,
while $\pi^{-1}(\mathcal{B}(\mathbb{T}))=\{A\times \mathbb{T}\in \mathcal{B}(\mathbb{T}^2) : A\in  \mathcal{B}(\mathbb{T})\}\subsetneq \mathcal{I}$.
\end{exemple}

\subsection{Convergence speed with the real Rajchman property}

We are now interested by the speed of convergence of the conditional correlations. Next result shows that even in $C^\infty$ regularity, the order of convergence is limitated by the Rajchman order.
\bigbreak
\begin{proposition}[Speed roof]
\label{roofsp}
Let $0\neq\xi\in\mathbb{Z}^d$ and $U\in \mathcal{A}$.
For all $\gamma>r\left(\left(m_{\xi,U}\right)_{\vert \mathbb{R}^*}\right)$, there exists $(f_1,f_2)\in (\mathcal{C}^\infty(\reciproque{\pi}(U))\cap \mathbb{L}^2_\mu(\Omega))^2$ such that the decay of conditional correlations is not faster than $t^{-\gamma}$, that is, there exists $C>0$ such that,
\[\limsup_{t\to+\infty}\abso{\mathbb{E}_\mu(Cov_{t}(f_1\cdot \mathds{1}_U\circ \pi,f_2\vert \mathcal{I}))}t^\gamma>C.\]
\end{proposition}

\begin{proof}
Let $U\in \mathcal{A}$, $\xi\in \mathbb{Z}^d\setminus\{0\}$.
Let $\gamma>r\left(\left(m_{\xi,U}\right)_{\vert \mathbb{R}^*}\right)$.
Consider \[ f : z\in\Omega \mapsto e^{2i\pi \langle\xi\vert (\pi_2\circ\psi_{U}^{-1})(z)\rangle}\mathds{1}_{\pi^{-1}(U)}(z)\in \mathcal{C}^\infty(\pi^{-1}(U))\cap\mathbb{L}^2_\mu(\Omega).\]
We get \[\intInd{\Omega}{\overline{f}(f\circ g_t)}{\mu}=\intInd{U}{e^{2i\pi t\langle\xi\vert v_{U}(x)\rangle}}{\left(\mu_\pi\right)(x)}=\intInd{\mathbb{R}}{e^{2i\pi tz}}{m_{\xi,U}(z)}.\]
By optimality of $r\left(\left(m_{\xi,U}\right)_{\vert \mathbb{R}^*}\right)$, for all $C>0$, for all $T>0$, there exists $t>T$ such that
\[\abso{\intInd{\mathbb{R}}{e^{2i\pi tz}}{\left(m_{\xi,U}\right)_{\vert \mathbb{R}^*}}}>\frac{C}{t^\gamma}.\]
The result follows by definition of conditional correlations.
\end{proof}

\begin{remark}
When $\Omega \simeq M\times \mathbb{T}^d$, we can take $f_1,f_2$ globally defined of class $\mathcal{C}^\infty$ in the Proposition \ref{roofsp}.
\end{remark}

\subsection{Speed of decay of conditional correlations for absolutely continuous measures}

We assume that $\Omega=M\times \mathbb{T}^d$ is the trivial bundle, endowed with an absolutely continuous measures. The speed of decay will depends on the  regularity properties of the velocity vector~$v$. We will use stationnary phase method \cite{Swor} that allow us to evaluate in an optimal way oscillating integrals. The regularity of the velocity vector and the presence of critical points influences the convergence order.

Before this study, we recall that by an immediate consequence of the Morse lemma, the set of critical points of the velocity vector is discrete.

\begin{lemme} [Isolation of non-degenerated critical points]
Let $M$ be a finite dimensional $\mathcal{C}^2$ manifold and $v\in \mathcal{C}^2(M,\mathbb{R})$.
Then every non-degenerated critical point of $v$ is isolated.
\end{lemme}

\begin{theorem}[Stationary phase, e.g. \cite{Swor}]
Let $\varphi\in\mathcal{C}^2(\mathbb{R}^n,\mathbb{R})$ with a unique critical point $x_c$.
We suppose that $x_c$ is non degenerated, in other words $det(Hess(x_c))\ne 0$.

Therefore, for any $a\in\mathcal{C}^1_0(\mathbb{R}^n,\mathbb{R})$, for all $t>0$, \[\intInd{\mathbb{R}^n}{e^{2i\pi t\varphi(x)}a(x)}{\lambda(x)}=\frac{a(x_c)e^{i\frac{\pi}{4}\left(\sum_{\lambda\in \sigma(Hess(x_c))}sgn(\lambda)\right)}}{t^{\frac{n}{2}}\sqrt{\abso{det(Hess(x_c))}}}+O_{+\infty}\left(\frac{1}{t^n}\right).\]

\end{theorem}

In the following we say that $a$ is a critical point of order  $q$ of a function $f$ if for $1\le m\le q$, $f^{(m)}(a)=0$ and $f^{(q+1)}(a)\neq0$.

We recall that for a probability measure $\mathbb{P}$, we note $\widehat{\mathbb{P}}$ its characteristic function.

\begin{lemme}[Convergence order with singular points in dimension $(1,d)$ with $d\geq 2$]
\label{Convdim}
Suppose that $dim(M)=1\le d$ and $M$ is a compact manifold of class $\mathcal{C}^\infty$.

Let $\xi\ne 0_{\mathbb{Z}^d}$ and $\ell\ge2$.

We suppose that $v$ is of class $\mathcal{C}^{\ell}$, that there exists a unique critical point of order $\ell-1$ for $\langle\xi\vert v(\cdot)\rangle$, and that all the other eventual critical points are of order strictly smaller.

We have
\[r\left(\widehat{\lambda}_{\langle\xi\vert v(\cdot)\rangle}\right)=\frac{1}{\ell}.\]

\end{lemme}

\begin{proof}
$M$ is compact Hausdorff, so the number of critical points for functions $\langle\xi\vert v\rangle$ is finite.
Let $(a_k)_{k\in\intervEnt{1}{m}}$ the family of critical point of $\langle\xi\vert v(\cdot)\rangle$ with respectives orders $l_k-1\geq 1$, $1\leq m<l_k,$
\[\langle\xi\vert v^{(m)}(a_k)\rangle=0\, and \, \langle\xi\vert v^{(l_k)}(a_k)\rangle\ne 0.\]
Let $(U_k)_{k\in \intervEnt{1}{m}}$ charts such that for all $k\in\intervEnt{1}{m}$, $a_k\in U_k$. We have
\[\langle\xi\vert v(\varphi_{U_k}^{-1}(x))\rangle=\langle\xi\vert v(a_k)\rangle+\frac{(x-\varphi_{U_k}(a_k))^{l_k}}{l_k!}\langle\xi\vert v^{(l_k)}(a_k)\rangle+(x-\varphi_{U_k}(a_k))^{l_k} h(x-\varphi_{U_k}(a_k))\]
with $h : \mathbb{R}\to\mathbb{R}$ such that $h(x)\xrightarrow[x\to 0]{}h(0)=0.$
Let pose for $k\in \intervEnt{1}{m},$ \[w_k : x\in\mathbb{R} \mapsto (x-\varphi_{U_k}(a_k))\left(\frac{\langle\xi\vert v^{(l_k)}(a_k)\rangle}{l_k!}+h(x-\varphi_{U_k}(a_k))\right)^{\frac{1}{l_k}}.\]
Therefore \begin{equation}\label{eq:w_k_derivee}w_k^' : x\in\mathbb{R}\mapsto \left(\frac{\langle\xi\vert v^{(l_k)}(a_k)\rangle}{l_k!}+h(x-\varphi_{U_k}(a_k))\right)^{\frac{1}{l_k}}+(x-\varphi_{U_k}(a_k))\kappa(x-\varphi_{U_k}(a_k))\end{equation}
with $\kappa_k(x)=\frac{h^'(x)}{\left(\langle\xi\vert v^{(l)}(a_k)\rangle+h(x)\right)^{\frac{1-l_k}{l_k}}}$.
Applying \eqref{eq:w_k_derivee} to $x=\varphi_{U_k}(a_k)$, \[w_k^'(\varphi_{U_k}(a_k))=\left(\frac{\langle\xi\vert v^{(l_k)}(a)\rangle}{l_k!}\right)^{\frac{1}{l_k}}\ne 0.\]
We get a local inverse of $w_k$ on a neighbourhood $W_k$ of $0$.
Let $(V_k)_{k\in \intervEnt{1}{m+m^'}}$ a family of charts on $M$ such that, for all $k\in \intervEnt{1}{m}$,
\begin{equation}\label{eq:chart_cover} (V_k = w_k^{-1}(W_k)\subset U_k\,and\,\forall j\in \intervEnt{m+1}{m^'},a_k\not\in V_k).\end{equation}
Let $(\psi_k)_{k\in\intervEnt{1}{m+m^'}}$ a partition of unity subordinated to $(V_k)_{k\in\intervEnt{1}{m+m^'}}$.
We get immediately that for $k\in \intervEnt{1}{m}$, $\psi_k(a_k)=1$.
By local reverse, for $k\in \intervEnt{1}{m}$, we get that on $W_k$, for all $x\in W_k$,
\begin{equation}\label{eq:local_inverse}\langle\xi\vert v(\varphi_{U_k}^{-1}(w_k^{-1}(x)))\rangle=\langle\xi\vert v^{(l_k)}(a_k)\rangle+x^{l_k}.\end{equation}
Writing $\check{\lambda}$ the push forward measure by charts of Lebesgue measure on the manifold $M$, for $k\in\intervEnt{1}{m}$, we get \begin{equation}\label{eq:change_manifold}\intInd{\varphi_{U_k}^{-1}(w^{-1}(W_k))}{e^{2i\pi t\langle\xi\vert v(x)\rangle}\psi_k(x)}{\check{\lambda}(x)}=e^{2i\pi t\langle\xi\vert v(\varphi_{W_k}^{-1}(w^{-1}(a_k)))\rangle}\intInd{W_k}{e^{2i\pi tx^{l_k}}\check{\psi_k}(x)}{\lambda(x)}\end{equation}
with \[\check{\psi_k} : x\in W_k\mapsto \psi\left(\varphi_{U_k}^{-1}(w_k^{-1}(x))\right)J(\varphi_{U_k}^{-1})(w_k^{-1}(x))J(w_k^{-1}(x))(x)\] and $J$ the Jacobian.
We know that $\check{\psi_k}$ is subordinated to $(V_j)_{j\in \intervEnt{1}{m+m^'}}$. Morevoer, by \eqref{eq:chart_cover}, the support of $\check{\psi_k}$ is included in $U_k$.
Thus, by \eqref{eq:change_manifold}, for all $t\ne 0$,
\[\begin{array}{ll}
\intInd{\varphi_{U_k}^{-1}(w_k^{-1}(W_k))}{e^{2i\pi t\langle\xi\vert v(x)\rangle}\psi(x)}{\check{\lambda}(x)}&=e^{2i\pi t\langle\xi\vert v(\varphi_{U_k}^{-1}(w_k^{-1}(a_k)))\rangle}\frac{1}{t^{\frac{1}{l_k}}}\intInd{\mathbb{R}}{e^{2i\pi x^{l_k}}\check{\psi_k}\left(\frac{x}{t^{\frac{1}{l_k}}}\right)}{\lambda(x)}.
\end{array}\]
Integrating by parts $u(x)v^'(x)$ with $u(x)=\frac{\check{\psi_k}\left(\frac{x}{t^{\frac{1}{l_k}}}\right)}{2il_k\pi x^{l_k-1}}$ and $v(x)=e^{2i\pi x^{l_k}}-1$, 
\[
\begin{split}
&\intInd{\mathbb{R}}{e^{2i\pi x^{l_k}}\check{\psi_k}\left(\frac{x}{t^{\frac{1}{l_k}}}\right)}{\lambda(x)}\\
&=\frac{1}{2i\pi l_k}\left(\intInd{\mathbb{R}}{\frac{(l_k-1)(e^{2i\pi x^{l_k}}-1)}{x^{l_k}}\check{\psi_k}\left(\frac{x}{t^{\frac{1}{l_k}}}\right)}{\lambda(x)}-\frac{1}{t^{\frac{1}{l_k}}}\intInd{\mathbb{R}}{\frac{(e^{2i\pi x^{l_k}}-1)}{x^{l_k-1}}\check{\psi_k}^'\left(\frac{x}{t^{\frac{1}{l_k}}}\right)}{\lambda(x)}\right).
\end{split}
\]
We know that for all $\alpha>0,$ \begin{equation}\label{eq:bound_psi_derivee}x\in \mathbb{R}\mapsto {\check{\psi_k}^'(x)}{(1+\abso{x}^\alpha)},x\in \mathbb{R}\mapsto {\check{\psi_k}(x)}{(1+\abso{x}^\alpha)}\text{ are bounded by a constant $C>0$.}\end{equation} 
Using the development at $x=0$ of $x\in\mathbb{R}\mapsto \frac{(e^{2i\pi x^{l_k}}-1)}{x^{l_k-1}}$, we get that \begin{equation}\label{eq:dvpt_bound}x\in\mathbb{R}\mapsto {\frac{(e^{2i\pi x^{l_k}}-1)}{x^{l_k-1}}}{(1+\abso{x}^{l_k-1})}\text{ is bounded by a constant $C_2>0$}.\end{equation}
For all $t\ne 0$ such that $\abso{t}>1$, by \eqref{eq:bound_psi_derivee}, for all $\alpha>0$,
\begin{equation}
    \label{eq:bound_psi} x\in \mathbb{R}\mapsto {\check{\psi_k}^'\left(\frac{x}{t^{\frac{1}{l_k}}}\right)}\frac{(1+\abso{x}^\alpha)}{t^{\frac{1}{l_k}}}\text{ is bounded by constant $C_3>0$}.
\end{equation}
Thus, by \eqref{eq:bound_psi_derivee} there exists $C_4>0$ such that for all $x\in\mathbb{R}$, for all $t\ne 0$, \begin{equation}\label{eq:bound_Hess}\abso{\frac{(l_k-1)(e^{2i\pi x^{l_k}}-1)}{x^{l_k}}\check{\psi_k}\left(\frac{x}{t^{\frac{1}{l_k}}}\right)}{(1+\abso{x}^{l_k})}\leq C_4\norme{\psi}_\infty.\end{equation}

By Lebesgue convergence theorem,\eqref{eq:dvpt_bound},\eqref{eq:bound_psi} and \eqref{eq:bound_Hess}
\[\intInd{\mathbb{R}}{e^{2i\pi x^{l_k}}\check{\psi_k}\left(\frac{x}{t^{\frac{1}{l_k}}}\right)}{\lambda(x)}\xrightarrow[t\to\pm\infty]{}\frac{\check{\psi_k}(0)}{2i\pi l_k}\intInd{\mathbb{R}}{\frac{(l_k-1)(e^{2i\pi x^{l_k}}-1)}{x^{l_k}}}{\lambda(x)}.\]
For $k\in \intervEnt{m+1}{m+m^'},$ by Greene formula, we get \[\intInd{\varphi_{U_k}^{-1}(V_k)}{e^{2i\pi t\langle\xi\vert v(x)\rangle}\psi_k(x)}{\check{\lambda}(x)}\in O_{\pm\infty}\left(\frac{1}{t}\right).\]
Thus, \[\intInd{M}{e^{2i\pi t\langle\xi\vert v(x)\rangle}}{\check{\lambda}(x)}=\sum_{k=1}^m\frac{1}{t^{\frac{1}{l_k}}}\left(J_{\varphi_{U_k}^{-1}}\left(\varphi_{U_k}(a_k)\right)\left(\frac{l_k!}{\langle\xi\vert v^{(l_k)}(a_k)\rangle}\right)^{\frac{1}{l_k}}I_{l_k}+o_{\pm\infty}(1)\right)+O_{\pm\infty}\left(\frac{1}{t}\right)\]
with $I_{l} :=\frac{1}{2i\pi l}\intInd{\mathbb{R}}{\frac{e^{2i\pi x^{l}}-1}{x^{l}}}{\lambda(x)}.$ A simple analysis shows that $I_l$ does not vanish\footnote{For $l$ even $I_l\in(-\infty,0)$ trivially, while for $l$ odd we have $I_l\in i(0,\infty)$, using a decomposition of the integral on intervals $[2n,2n+2]$.}.

By hypothesis, there exists just one critical $a$ point with maximal order $\ell$, hence
 \[\intInd{M}{e^{2i\pi t\langle\xi\vert v(x)\rangle}}{\check{\lambda}(x)}= \frac{1}{t^{\frac{1}{\ell}}}\left(J_{\varphi_{U_j}^{-1}}\left(\varphi_{U_j}(a)\right)\left(\frac{\ell!}{\langle\xi\vert v^{(\ell)}(a)\rangle}\right)^{\frac{1}{\ell}}I_\ell+o_{\pm\infty}(1)\right).\]
Thus, \[r\left(\widehat{\lambda}_{\langle\xi\vert v(\cdot)\rangle}\right)=\frac{1}{\ell}.\]
\end{proof}

We get then the next proposition.

\begin{proposition}
    Under the hypothesis of Lemma \ref{Convdim}, we get that the order of decay of correlation $\gamma$ with $f_1,f_2\in\mathcal{C}^\infty(\Omega)$ satisfies \[\gamma \leq \frac{1}{\ell}.\] 
\end{proposition}

\begin{proof}
    In proof of Lemma \ref{Convdim}, you can take $f_1=e^{2i\pi \langle\xi\vert v(\cdot)\rangle}$ and $f_2 = \psi_j$ with $j$ the index of the chart containing the critical point $a$ with a maximal order $\ell$.
\end{proof}

We can now treat the general case.

The following lemma is obtained with Theorem 7.5 p.226 in \cite{SiDM}.

\begin{lemme}[Convergence order around a singular points with analytical functions]

We will use notations and results of \cite{SiDM}.

We place ourselves in $\mathbb{R}^n$.

Let $\xi\ne 0_{\mathbb{Z}^d}$.

Let $v$ be such that  $\langle\xi\vert v\rangle$ is analytic with $0$ a critical point of multiplicity $m\in\mathbb{N}^*$

Then there exist $\alpha\in\mathbb{Q}^*_+$ and $j\in\mathbb{N}$ in a compact  neighbourhood $V$ of $0$ such that for all $\varphi\in \mathcal{D}_V(\Omega)$,

\[\exists C\in\mathbb{R}^*,
{\intInd{\mathbb{R}^n}{e^{2i\pi t\langle\xi\vert v(x)\rangle}\varphi(x)}{\lambda(x)}}
\sim {\frac{C(\ln(t))^j}{t^\alpha}}.\]
\end{lemme}

\begin{definition}[Logarithmic convergence order]
Consider a probability measure $\nu$  on $(\mathbb{R},\mathcal{B}(\mathbb{R}))$

We call logarithmic convergence order of $\nu$ the value

\[rl(\nu) :=\inf\left\{\beta >0 : \exists C>0,\forall t>0,\abso{\widehat{\mu(t)}}\leq\frac{C\abso{\ln(t)}^\beta}{t^{r(\nu)}}\right\}.\]
\end{definition}

Next result may be obtained by a partition of unity and the previous lemma.
\begin{proposition}[Convergence order with singular points in dimension $(n,d)$ with $d\geq 2$]

Suppose that $M$ is analytic.
Let $\xi\ne 0_{\mathbb{Z}^d}$.
Suppose that $v$ such that  $\langle\xi\vert v\rangle$ is analytic.
We have that \[r\left(\lambda_{\langle\xi\vert v(\cdot)\rangle}\right)\in \mathbb{Q}_+^*\text{ and }rl\left(\lambda_{\langle\xi\vert v(\cdot)\rangle}\right)\in\mathbb{N}.\]
\end{proposition}

\begin{coro}
Let $v$ be a $\mathcal{C}^2$ function.
If there exists only one critical points $a\in M$ and it satisfies
 \[\exists\xi\ne 0_{\mathbb{Z}^d}, \left(\nabla(\langle\xi\vert v(\cdot)\rangle)(a)=0\, and \, det(Hess\langle\xi\vert v(\cdot)\rangle(a))\ne 0\right)\]
then $r(\widehat{\lambda}_{\langle\xi\vert v(\cdot)\rangle})= \frac{n}{2}.$
\end{coro}

\begin{proof}
Let $\xi\ne 0_{\mathbb{Z}^d}$.
Applying the stationary phase theorem on respectives charts $(U,\varphi_U)$ of $\mathcal{A}$ containing at most just one critical point $a\in U$ and test functions $\psi$ on $U$, we get 

\[
\begin{split}
\intInd{M}{e^{2i\pi t\langle\xi\vert v(x)\rangle}\psi(x)}{\widehat{\lambda}(x)}&=\intInd{\mathbb{R}^n}{e^{2i\pi t\langle\xi\vert v\left(\varphi_U^{-1}(x)\right)\rangle}\psi\left(\varphi_U^{-1}(x)\right)J(\varphi_U^{-1})(x)}{\lambda(x)}\\
&=\frac{\psi\left(\varphi_U^{-1}(a)\right)e^{i\frac{\pi}{4}\left(\sum_{\lambda\in \sigma(Hess(a))}sgn(\lambda)\right)}}{t^{\frac{n}{2}}\sqrt{\abso{det(Hess(a))}}}+O_{+\infty}\left(\frac{1}{t^n}\right).
\end{split}
\]
By partition of unity we obtain
\[{\exists C\in\mathbb{R},\intInd{M}{e^{2i\pi\langle\xi\vert v(x)\rangle}}{\widehat{\lambda}(x)}=\frac{C}{t^{\frac{n}{2}}}+O_{+\infty}\left(\frac{1}{t^n}\right)}.\]
\end{proof}
\begin{remark} Note that the result is consistent with the case $n=1$.
\end{remark}

We recall the definition of Damien Thomine \cite{DaTho} 
\begin{definition}[Anisotropic Sobolev space on $\mathbb{R}\times \mathbb{T}$]
    Let $s\geq 0$.
    Let $h : x\in \mathbb{R}\mapsto \sqrt{1+x^2}$. Let
    \[H^{s,0}(\mathbb{R}\times\mathbb{T}) := \left\{ f\in \mathbb{L}^2(\mathbb{R}\times\mathbb{T}) : \sum_{k\in\mathbb{Z}}\intInd{\mathbb{R}}{\abso{\widehat{f}(x,k)}^2h^{2s}(x)}{\lambda(x)}\in\mathbb{R}\right\}.\]
\end{definition}

\begin{proposition}
Consider $(\mathbb{R}\times\mathbb{T},\mu\otimes \lambda,(g_t)_{t\in \mathbb{R}})$ such that $r=r(\mu)>0$ and for $(x,y)\in \mathbb{R}\times\mathbb{T},g_t(x,y)=(x,y+tx)$.
Let $s>\frac{1}{2}$.
Let for $\varepsilon\in\left]0,\frac{1}{2s}\right[,q_\varepsilon := \min{\left\{s(1-\varepsilon),r(\mu)\right\}}$.
Then, we get for all $\varepsilon \in\left]0,\frac{1}{2s}\right[$, there exists $C_{s,\varepsilon}>0$ such that for all $t>0$, 
\[\abso{\mathbb{E}_{\mu\otimes\lambda}(Cov_t(f_1,f_2\vert \mathcal{I}))}\leq \frac{C_{s,\varepsilon}\norme{f_1}_{H^{s,0}(\mathbb{R}\times\mathbb{T})}\norme{f_2}_{H^{s,0}(\mathbb{R}\times\mathbb{T})}}{t^{q_\varepsilon}}
\]
and if $\gamma>0$ denotes the convergence order on $H^{s,0}(\mathbb{R}\times \mathbb{T})$, we get $\min{\left\{s-\frac{1}{2},r(\mu)\right\}}\leq \gamma$. Moreover if $supp(\mu)$ is compact, then $\gamma\leq r(\mu)$.
\end{proposition} 

\begin{proof}
Let $s>\frac{1}{2}$ and $(f_1,f_2)\in (H^{s,0}(\mathbb{R}\times\mathbb{T}))^2$.
Then by Cauchy-Schwarz inequality and Parseval \begin{equation}\label{eq:CS_control}\begin{array}{ll}\abso{\mathbb{E}_{\mu\otimes\lambda}(Cov_t(f_1,f_2\vert \mathcal{I}))} &=\abso{\sum_{k\in\mathbb{Z}^*}\intInd{\mathbb{R}^2}{\overline{\widehat{f_1}}(u,k)\widehat{f_2}(v,k)\widehat{\mu}(kt-(u-v))}{\lambda(u,v)}}\\
&\leq S(t)\norme{f_1}_{H^{s,0}(\mathbb{R}\times\mathbb{T})}\norme{f_2}_{H^{s,0}(\mathbb{R}\times\mathbb{T})}\end{array}\end{equation}
with \[S(t) := \sup_{k\in \mathbb{Z}^*} C_\mu\left(\intInd{\mathbb{R}^2}{\frac{1}{(1+\abso{kt-(u-v)})^{2r(\mu)}h^{2s}(u)h^{2s}(v)}}{\lambda(u,v)}\right)^\frac{1}{2}.\]
We can consider $k=1$ considering $kt$ instead of $t$.
Let $D_1(t) := \left\{(u,v)\in \mathbb{R}^2 : \abso{t-(u-v)}\leq 2t\ and\ \abso{v}\leq \frac{t}{4}\right\}$, $D_2(t) := \left\{(u,v)\in\mathbb{R}^2 : \abso{t-(u-v)}\leq 2t\ and\ \abso{v}>\frac{t}{4}\right\}$ and $D_3(t) :=\left\{(u,v)\in\mathbb{R}^2 : \abso{t-(u-v)}> 2t\right\}$.

We get that for $(u,v)\in D_1(t)$
\begin{equation}
\label{eq:D_1}
    \abso{u}\geq \frac{3t}{4}\text{ and }h^{2s}(u)\geq \frac{16^s}{9^st^{2s}}
\end{equation}
and for $(u,v)\in D_2(t)$,
\begin{equation}
\label{eq:D_2}
    h^{2s}(v)> \frac{16^s}{t^{2s}}.
\end{equation}

Using that $s>\frac{1}{2}$, for all $\varepsilon\in \left]0,\frac{1}{2s}\right[$,
\begin{equation}\label{eq:control_h}
    \intInd{\mathbb{R}}{h^{-2s\varepsilon}(u)}{\lambda(u)}<+\infty
\end{equation}
and by \eqref{eq:D_1} and \eqref{eq:D_2}, for and
\begin{equation}\label{eq:int_1_control}\intInd{D_1(t)}{\frac{1}{(1+\abso{t-(u-v)})^{2r(\mu)}h^{2s}(u)h^{2s}(v)}}{\lambda(u,v)}\leq\frac{16^{s(1-\varepsilon)}}{{(3t)}^{2s(1-\varepsilon)}}\intInd{\mathbb{R}^2}{h^{-2s\varepsilon}(u)h^{-2s}(v)}{\lambda(u,v)},\end{equation}
\begin{equation}\label{eq:int_2_control}
    \intInd{D_2(t)}{\frac{1}{(1+\abso{t-(u-v)})^{2r(\mu)}h^{2s}(u)h^{2s}(v)}}{\lambda(u,v)}\leq\frac{16^{s(1-\varepsilon)}}{t^{2s(1-\varepsilon)}}\intInd{\mathbb{R}^2}{h^{-2s}(u)h^{-2s\varepsilon}(v)}{\lambda(u,v)}.
\end{equation}
Therefore, by definition of $D_3(t)$, \eqref{eq:int_1_control} and \eqref{eq:int_2_control},
\begin{equation}
    \intInd{\mathbb{R}^2}{\frac{1}{(1+\abso{t-(u-v)})^{2r(\mu)}h^{2s}(u)h^{2s}(v)}}{\lambda(u,v)}\leq \frac{D_{s,\varepsilon}^2}{t^{2q_\varepsilon}}\intInd{\mathbb{R}^2}{h^{-2s}(u)h^{-2s\varepsilon}(v)}{\lambda(u,v)}
\end{equation}
with $D_{s,\varepsilon} = \max{\{\frac{1}{2^{r(\mu)}},{4}^{s(1-\varepsilon)}\sqrt{2}\}}$ and $q_\varepsilon=\min{\{r(\mu),s(1-\varepsilon)\}}$.
Thus, \begin{equation}\label{eq:control_S}\forall \varepsilon \in\left]0,\frac{1}{2s}\right[,S(t)\leq \frac{C_{\mu}C_{s,\varepsilon}}{t^{q_\varepsilon}}\left(\intInd{\mathbb{R}}{h^{-2s}(u)}{\lambda(u)}\intInd{\mathbb{R}}{h^{-2s\varepsilon}(v)}{\lambda(v)}\right)^{\frac{1}{2}}.\end{equation}

By \eqref{eq:CS_control},\eqref{eq:control_h} and \eqref{eq:control_S},
\[\abso{\mathbb{E}_{\mu\otimes\lambda}(Cov_t(f_1,f_2\vert \mathcal{I}))}\leq\frac{C_{s,\varepsilon}}{t^{q_\varepsilon}}\norme{f_1}_{H^{s,0}(\mathbb{R}\times\mathbb{T})}\norme{f_2}_{H^{s,0}(\mathbb{R}\times\mathbb{T})}.\]
Moreover, for $\gamma>0$, the convergence order on $H^{s,0}(\mathbb{R}\times \mathbb{T})$, we get $\min{\left\{s-\frac{1}{2},r(\mu)\right\}}\leq \gamma$ and when $supp(\mu)$ is compact, $\gamma\leq r(\mu)$.
\end{proof}

\begin{remark}
    When $(f_1,f_2)\in (H^{s,0}(\mathbb{R}\times \mathbb{T}))^2$ with $s\geq \frac{r(\mu)+1}{2}$, we get that $\gamma=r(\mu)$ as optimal order of convergence when $supp(\mu)$ is compact.
\end{remark}

\begin{remark}
    For $s>\frac{1}{2}$, for all integer $k>s+\frac{3}{2}$, $\mathcal{C}^k_c(\mathbb{R}\times\mathbb{T})\subset H^{s,0}(\mathbb{R}\times \mathbb{T})$, and then, for $k>s+\frac{3}{2}$ the convergence order $\gamma>0$ for $\mathcal{C}^k_c(\mathbb{R}\times\mathbb{T})$ is such that $\gamma \in \left[\min{\left\{s-\frac{1}{2},r(\mu)\right\}},r(\mu)\right]$ when $supp(\mu)$ is compact.
\end{remark}

\begin{remark}
    When $\mu=\lambda_v$ with $v\in \mathcal{C}^\infty$ with a unique critical point of order $l-1$ with $l\geq 3$, then $\gamma=r(\mu)=\frac{1}{l}$.
\end{remark}

\section{Keplerian shear, Rajchman property and Diophantine approximation}

In this section we present an application of Keplerian shear (with speed estimates) to dynamical Borel Cantelli lemmas. The latter is linked to Diophantine approximation and we discuss the relations between the Rajchman property and Diophantine properties.

Let the probabilistic space $(\Omega,\mu)$.
The main application consists in readjusting the dynamic {Borel-Cantelli} theorem in the non-ergodic case.
We find an adaptation of {Sprindzuk}'s theorem inspired by the thesis of {Victoria Xing} \cite{VicXing}.
The Keplerian shear will ensure that for almost any $x$, there exists an infinity of integers $n$ satisfying $T^n(x)\in A_n$.

\begin{theorem}[Variable {Sprindzuk}]
\label{Sprinvar}
Let $(f_k)_{k\in\mathbb{N}^*}$ and $(g_k)_{k\in\mathbb{N}^*}$ be two sequences of measurable and non negative functions.
Let $(\varphi_k)_{k\in\mathbb{N}^*}$ a real sequence such that for all $k\in\mathbb{N}^*$, $0\leq g_k\leq \varphi_k\leq 1\,\mu-a.e$.
Let $\delta>1$ and $C>0$.
Suppose that for all $(m,n)\in\left(\mathbb{N}^*\right)^2$ satisfying $n\geq m$, 
\[\intInd{\Omega}{\left(\sum_{k=m}^nf_k(x)-g_k(x)\right)^2}{\mu(x)}\leq C\left(\sum_{k=m}^n\varphi_k\right)^\delta.\]
Then, for all $n\in\mathbb{N},$ for all $\varepsilon>0$,
\[\sum_{k=1}^n f_k = \sum_{k=1}^n g_k+O\left(\phi(n)^{\frac{\delta}{2}}\left(\log\left(\phi(n)\right)\right)^{1+\varepsilon}\right)\, \mu-a.e.\]
with $\phi(n)= \sum_{k=1}^n\varphi_k$.
\end{theorem}

For the proof, we will elaborate on the proof of the original theorem in the book by {Sprindzuk} \cite{Sprdz}, page 45 formula (68).
\begin{proof}
Let $\delta>1$ and let us denote for $I\subset \mathbb{N}^*,\phi(I)=\sum_{k\in I}\varphi_k$.
By the fact that for all $ k\in\mathbb{N}^*$, $0\leq \varphi_k\leq 1$, we have that there exists $(n_v)_{v\in\mathbb{N}^*}\in{\mathbb{N}^*}^{\mathbb{N}^*}$ such that for all $v\in\mathbb{N}^*$,
\[\phi(n_v)<v\leq \phi(n_v+1).\]
We have also for all $v\in\mathbb{N}^*$,\[n_{v+1}\geq n_v+1\]
and for all $v\in\mathbb{N}^*$,\[\phi(n_v)+1<v+1\leq \phi(n_{v+1}+1).\]
Thus, for all  $u,v\in\mathbb{N}^*$,\[\left(u<v\Longrightarrow \intervEnt{n_u+1}{n_v}\ne\emptyset\right).\]
Let \[\begin{array}{lcl}
\sigma : &\mathcal{P}(\mathbb{N}^*)&\to \mathcal{P}(\mathbb{N}^*)\\
&I&\mapsto \left\{n_w\in\mathbb{N} : w\in I\right\}.\end{array}\]
Let for $r\in\mathbb{N}^*, s\in\intervEnt{0}{r}$ sets of parts \[J_{r,s} :=\left\{\intervEnt{i2^s+1}{(i+1)2^s}\in\mathcal{P}(\mathbb{N}^*) : i\in\intervEnt{0}{2^{r-s}-1}\right\}.\]
Let $r\in\mathbb{N}^*$ and $s\in\intervEnt{0}{r}.$
We notice that \[\union{I\in J_{r,s}}{\sigma(I)}=\intervEnt{1}{n_{2^r}}.\]
Let $i\in \intervEnt{0}{2^{r-s}-1}.$
We have that \[\phi\left(n_{i2^s}\right)<i2^s\leq \phi\left(n_{i2^s}+1\right)\leq \phi\left(n_{i2^s}\right)+1.\]
Therefore \[i2^s-1\leq \phi\left(n_{i2^s}\right)<i2^s,\]
\[\phi\left(n_{(i+1)2^s}\right)<(i+1)2^s\leq \phi\left(n_{(i+1)2^s}+1\right)\leq \phi\left(n_{(i+1)2^s}\right)+1\]
and \[(i+1)2^s-1\leq\phi\left(n_{(i+1)2^s}\right)<(i+1)2^s.\]
Thus, \[\phi\left(\sigma\left(\intervEnt{i2^s+1}{(i+1)2^s}\right)\right)=\phi\left(n_{(i+1)2^s}\right)-\phi\left(n_{i2^s}\right)\leq 2^s+1\leq 2^{s+1}.\]
Therefore, \[\sum_{I\in J_{r,s}}\left(\phi(\sigma(I))\right)^\delta\leq 2^\delta 2^{r+s(\delta-1)}.\]
Denote \[J_r :=\union{s\in \intervEnt{0}{r}}{J_{r,s}}.\]
Thus, we have \[\sum_{I\in J_{r}}\left(\phi(\sigma(I))\right)^\delta=\sum_{s=0}^r\sum_{I\in J_{r,s}}\left(\phi(\sigma(I))\right)^\delta\leq \frac{2^{2\delta-1} 2^{r\delta}}{2^{\delta-1}-1}.\]
Let \[\begin{array}{lcl}
h : & \mathbb{N}^*\times \Omega&\to \mathbb{R}\\
& (l,x)&\mapsto \sum_{I\in J_l}\left(\sum_{k\in I}f_k(x)-g_k(x)\right)^2.\end{array}\]
Therefore, \[\intInd{\Omega}{h(r,x)}{\mu(x)}\leq C\left(\sum_{I\in J_{r}}\phi(\sigma(I))\right)\leq C_1 2^{r\delta}\]
for some constant $C_1$.
Let $\varepsilon>0$.
By the {Markov} inequality,
\[\mu\left(h(r,X)\geq C_1r^{1+\varepsilon}2^{r\delta}\right)\leq r^{-(1+\varepsilon)}.\]
By {Borel-Cantelli},
\[\mu-a.e \,x\in \Omega,\exists r_x\in\mathbb{N}^*,\forall r\geq r_x,h(r,x)\leq C_1r^{1+\varepsilon}2^{r\delta}.\]
For $v\in\mathbb{N}^*$, we have that $\intervEnt{1}{v}$ can be split into a finite number $r_v$ of interval $J_r$ such that \[r_v\leq \lfloor{\log_2(v)}\rfloor+1.\]
Let $J(v)$ denote the set of these intervals.
We then get \[\sum_{k=1}^{n_v}f_k-g_k=\sum_{I\in J(v)}\left(\sum_{k\in I}f_k-g_k\right).\]
By the {Cauchy-Schwarz} inequality,
\[\mu-a.e\,x\in \Omega,\left(\sum_{k=1}^{n_v}f_k-g_k\right)^2\leq r_v\sum_{I\in J(v)}\left(\sum_{k\in I}f_k-g_k\right)^2=r_vh\left(r_v,x\right).\]
Thus
\[\mu-a.e\,x\in\Omega,\exists r_x\in\mathbb{N}^*,\forall r\geq r_x,\left(\sum_{k=1}^{n_v}f_k(x)-g_k(x)\right)^2\leq C_1r^{1+\varepsilon}2^{r\delta}\leq C_1\log_2^{2+\varepsilon}(v)v^{\delta}.\]
Therefore, \[\mu-a.e\,x\in\Omega,\exists r_x\in\mathbb{N}^*,\forall r\geq r_x,\abso{\sum_{k=1}^{n_v}f_k(x)-g_k(x)} \leq C_1^{\frac{1}{2}}\log_2^{1+\frac{\varepsilon}{2}}(v)v^{\frac{\delta}{2}}.\]
Let $n\in\mathbb{N}^*.$
Taking $v\in\mathbb{N}^*$ such that $n_v\leq n\leq n_{v+1}$ we get
 \[\mu-a.e\,x\in\Omega,\sum_{k=1}^{n_v}f_k(x)\leq \sum_{k=1}^{n}f_k(x)\leq \sum_{k=1}^{n_{v+1}}f_k(x).\]
Moreover, \[\mu-a.e\,x\in \Omega,0\leq \sum_{k=n_v}^{n_{v+1}}g_k(x)\leq \phi\left(n_{v+1}\right)-\phi\left(n_v\right).\]
In addition, \[v-1\leq\phi\left(n_v\right)<v\leq \phi\left(n_v+1\right)\leq\phi\left(n_{v+1}\right)<v+1.\]
Then \[0<\phi\left(n_{v+1}\right)-\phi\left(n_v\right)< 2,\]
and \[\mu-a.e\,x\in \Omega, 0\leq \sum_{k=n_v+1}^{n_{v+1}}g_k(x)< 2.\]
Therefore 
\[\mu-a.e\,x\in\Omega,\sum_{k=1}^n f_k(x)=\sum_{k=1}^{n}g_k(x)+O\left(\log_2^{1+\frac{\varepsilon}{2}}\left(\phi(n)\right)\left(\phi(n)\right)^{\frac{\delta}{2}}\right).\]
\end{proof}

\begin{theorem} [Dynamical {Borel-Cantelli} by Keplerian shear]\label{thm:dbcbyks}
Suppose that $(\Omega,T,\mu)$ is a discrete measure preserving dynamical system with Keplerian shear. Let $(A_n)$ be a sequence of measurable sets and note $S_{M,N} = \sum_{k=M}^N\mathds{1}_{A_k}\circ T^k$, $S_N=S_{0,N}$.

We suppose that \ref{1.a} or \ref{1.b} holds and \ref{2}
\begin{enumerate}
\item
\begin{enumerate}
\item\label{1.a}
There exists $\gamma>0,C>0$ such that  for all $(j,k)\in\mathbb{N}^2$ satisfying $k\ne j$, 
$$\abso{\mathbb{E}\left(Cov_{k-j}\left(\mathds{1}_{A_j},\mathds{1}_{A_k}\vert\mathcal{I}\right)\right)}\leq \frac{C}{\abso{k-j}^\gamma}$$
\smallbreak
\par\item\label{1.b}
There exists $\gamma>0,D>0$ and $\mathcal{B}\subset \mathbb{L}^2_\mu(\Omega)$, a {Banach} space such that  there exists $C>0$ such that for all $f,g\in \mathcal{B}$, for all $n\in\mathbb{N}$, $\abso{\mathbb{E}(Cov_n(f,g\vert\mathcal{I}))}\leq \frac{C\norme{f}_\mathcal{B}\norme{g}_\mathcal{B}}{n^{\gamma}}$ and for all $j\in\mathbb{N}$, $\norme{\mathds{1}_{A_j}}_\mathcal{B}\leq D$
\end{enumerate}
\item\label{2}
There exists $\beta>\max(\frac12,1-\frac\gamma2)$ such that $\liminf_{N\to\infty}\left({\frac{\mathbb{E}(S_N\vert \mathcal{I})}{N^\beta}}\right)>0\ \mu-a.s$.
\end{enumerate}
We have  \[\frac{S_N}{\mathbb{E}(S_N\vert \mathcal{I})}\overset{\mu-a.s}{\xrightarrow[N\to+\infty]{}}1.\]
\end{theorem}
\begin{remark}
We emphasize that in contrast to usual dynamical Borel Cantelli lemmas, the evolution of the sum $S_N$ is only compared to the  conditional expectation $\mathbb{E}(S_N\vert \mathcal{I})$, instead of the classical expectation $E(S_N)$.
\end{remark}
\begin{proof}
Without loss of generality we assume that $\gamma\in(0,1)$. Note that \ref{1.b} implies \ref{1.a} so \ref{1.a} holds in both cases.

By expansion of the square as a  double sum and the series-integral comparison criterion, we get $\forall (M,N)\in\mathbb{N}^2$ s.t. $M<N$,
\[
\begin{split}
\norme{S_{M,N}-\mathbb{E}\left(S_{M,N}\vert \mathcal{I}\right)}_{2}^2&\leq  \mathbb{E}(S_{M,N})+2C\sum_{l=1}^{N-M+1} \frac{(N-M+1)}{l^\gamma}\\
&\le 
\mathbb{E}(S_{M,N}) + 2C (N-M+1)^{2-\gamma}.
\end{split}
\]
By setting $\delta = 2-\gamma$, and noticing that $E(S_{M,N})\le N-M$ we get that there exists $D>0$ satisfying
\[\forall N\in\mathbb{N},\norme{S_{M,N}-\mathbb{E}\left(S_{M,N}\vert \mathcal{I}\right)}_{2}^2\leq D(N-M+1)^{\delta}.\]
By variable Sprindzuk Theorem \ref{Sprinvar} with $\varphi_k=1$ for all $k$, we get \[S_N-\mathbb{E}(S_N\vert \mathcal{I})\overset{a.s}{=}  O(N^{\frac{\delta}{2}}\log^{1+\varepsilon}(N)).\]
Then 
\[\frac{S_N}{\mathbb{E}(S_N\vert \mathcal{I})}
=1+O\left(\frac{N^{\frac\delta2}\log^{1+\varepsilon}(N)}{N^\beta}\right)
\overset{a.s}{\xrightarrow[N\to+\infty]{}}1.\]
\end{proof}

\begin{exemple}
Let us place ourselves in the case of $(\mathbb{T}^2,\lambda\otimes\lambda,T)$ with \[Mat(T)=\begin{pmatrix}1 & 0\\ 1 & 1\end{pmatrix}.\]
Consider a sequence $(b_n)\in\mathbb{T}^\mathbb{N}$ and let \[A_n=\mathbb{T}\times \left[b_n-\frac{1}{n^p},b_n+\frac{1}{n^p}\right]\] with $0<p<\frac{1}{2}.$
We get that $\forall N\in\mathbb{N}^*$,
\[\mathbb{E}(S_N\vert \mathcal{I})=E(S_N)=2\sum_{k=1}^N\frac{1}{k^p}\sim \frac{2}{1-p}N^{1-p}.\]
Let $s:=\frac12$ and $n\in\mathbb{N}^*$. By direct computation of the {Fourier} coefficients of $1_{A_n}$, we see that \[\mathds{1}_{A_n}\in H^{s,0}(\mathbb{T}^2)\]
since \[\norme{\mathds{1}_{A_n}}_{H^{s,0}(\mathbb{T}^2)}^2\leq \frac{1}{\pi^2}\zeta(2)+\frac{2}{n^p}\leq  3.\]
By the decorrelation estimates in \cite{DaTho}, $\forall (k,j)\in\mathbb{N}^2$,
\[\mathbb{E}\left(Cov_{k-j}\left(\mathds{1}_{A_k},\mathds{1}_{A_j}\vert\mathcal{I}\right)\right)\leq\frac{4^s 3^2}{\abso{k-j}^{2s}}.\]
The hypotheses are then brought together, applying Theorem \ref{thm:dbcbyks} with $\gamma=1$ and $\beta\in(\frac12,1-p)$, we get 
\[S_N\sim \frac{2}{1-p}N^{1-p}\quad{\mu-a.s}.\]
\end{exemple}

In the case of any measure of {Rajchman}, we have as an application a proposition based on Theorem 4.2 page 551 of {Athreya} \cite{Atry} and also the theorem of {Kurzweil} (\cite{Kurz} and \cite{Tsen}, Theorem 1.3 page 3) concerning Diophantine numbers.
To introduce this theorem, we will also have to recall the Shrinking Target Property (STP) and its monotone version (MSTP) of a dynamical system.

\begin{definition}[Shrinking Target Properties]

A discrete dynamical system $(\Omega,\mu,T)$ is called STP if for all sequence of balls $(B_n)_{n\in\mathbb{N}}$ of radius $r_n>0$ tending to $0$ and satisfying \[\sum_{n\in\mathbb{N}}\mu(B_n)=+\infty\text{ we have } \mu\left(\overline{\underset{n\to+\infty}{\lim}}T^{-n}(B_n)\right)=1.\]

The MSTP property is defined in the same way, assuming in addition that the sequence $r_n$ is monotone.
\end{definition}

Let us now recall {Kurzweil}'s theorem in dimension $1$.

\begin{theorem}[{Kurzweil-Tseng}]

\label{kurztseng}

Consider the dynamical system $(\mathbb{T},\lambda,T_\alpha)$ with $\alpha\in\mathbb{T}$ and $ T_\alpha : x\mapsto x+\alpha$.
Then the dynamical system is $s-$MSTP if and only if $\alpha$ is $s$-Diophantine (i.e., $\alpha\in\mathcal{D}io(s)$, see definition \ref{def:s_diophante}).
\end{theorem}

We also deduce with the documents of {Chaika} \cite{ChaiKon} the following theorem.

\begin{theorem}[{Kurzweil}]
\label{kurzmstp}
For $\lambda-a.e \,\alpha\in [0,1[,(\mathbb{T},\lambda, T_\alpha)\, is \, MSTP$.
\end{theorem}

In the case of singular continuous {Rajchman} measures, we will have a weakening of the Theorem \ref{kurzmstp} of {Kurzweil}.
First, we present a dynamical Borel Cantelli for Rajchman measures. As in Theorem~\ref{thm:dbcbyks}, it is based on covariance estimates. Here they do not rely on the Keplerian shear but are  obtained directly from the Rajchman property. 

\begin{theorem}[Dynamical Borel-Cantelli for Rajchman measures]\label{thm:dbcbyrp}
Consider a Rajchman measure $\mu$ on $\mathbb{T}$ of order
 $0\le r(\mu)\le\frac{1}{2}$. Let $T$ be the transvection on $\mathbb{T}^2$ defined by $T(x,y)=(x,x+y)$.
Let $C>0$ and $s>\frac{1}{r(\mu)}-1$. 
We have \[\frac{S_N}{\mathbb{E}(S_N)}\xrightarrow[N\to+\infty]{\mu\otimes \lambda-a.e} 1\]
where
\[S_N=\sum_{k=1}^N \mathds{1}_{A_k}\circ T^k, \quad\text{and}\quad A_n = \mathbb{T}\times \left]z_n-\frac{C}{n^{\frac{1}{s}}},z_n+\frac{C}{n^{\frac{1}{s}}}\right[.\]
\end{theorem}

\begin{proof}
Let $p=\frac{1}{s}$.
Let for $(m,n)\in(\mathbb{N}^*)^2$ such that $m>n$ the sum \[S_{(m,n)}=\sum_{k=m}^n\mathds{1}_{A_k}\circ T^k.\]
By definition and using the invariance of the measure we get
\[\begin{array}{ll}\intInd{\mathbb{T}^2}{\left(S_{(m,n)}-\mathbb{E}\left(S_{(m,n)}\right)\right)^2}{\mu\otimes \lambda}&=\sum_{k=m}^n\sum_{j=m}^n\intInd{\Omega}{\mathds{1}_{A_k}\circ T^{k-j}\mathds{1}_{A_j}-\mathbb{E}_\mu\left(\mathds{1}_{A_k}\vert \mathcal{I}\right)\mathbb{E}_\mu\left(\mathds{1}_{A_j}\vert \mathcal{I}\right)}{\mu}\\
&\leq \mathbb{E}_{\mu}(S_{(m,n)})+2\somme{m}{n}{\somme{1}{n-m}{\intInd{\Omega}{\mathds{1}_{A_{j+l}}\circ T^{l}\mathds{1}_{A_j}-\mathbb{E}_\mu\left(\mathds{1}_{A_{j+l}}\vert \mathcal{I}\right)\mathbb{E}_\mu\left(\mathds{1}_{A_j}\vert \mathcal{I}\right)}{\mu}}{l}}{j} .\end{array}\]
Using the Fourier serie of the indicator function of an interval we obtain \[\begin{array}{ll}\abso{\intInd{\Omega}{\mathds{1}_{A_{j+l}}\circ T^{l}\mathds{1}_{A_j}-\mathbb{E}_\mu\left(\mathds{1}_{A_{j+l}}\vert \mathcal{I}\right)\mathbb{E}_\mu\left(\mathds{1}_{A_j}\vert \mathcal{I}\right)}{\mu}}&=\abso{\sum_{k\in\mathbb{Z}^*}\frac{\sin\left(\frac{2k\pi C}{j^p}\right)\sin\left(\frac{2k\pi C}{(j+l)^p}\right)}{k^2\pi^2}\widehat{\mu}(kl)}\\
&\leq 4C_\mu C\sum_{k\in\mathbb{N}^*}\frac{1}{l^{r(\mu)}}\frac{\abso{\sin\left(\frac{2\pi k C}{j^p}\right)}}{k^{1+r}(j+l)^p\pi}\\
&\simeq \frac{2^{2+r}C_\mu C^{1+r}\pi^{1-r}}{j^{pr(\mu)}(j+l)^pl^{r(\mu)}}\intInd{\mathbb{R}_+^*}{\frac{\abso{\sin(x)}}{x^{1+r}}}{\lambda(x)}.\end{array}\]

Combining this upper bound with the previous etsimate gives \[\begin{array}{ll}\intInd{\mathbb{T}^2}{\left(S_{(m,n)}-\mathbb{E}\left(S_{(m,n)}\right)\right)^2}{\mu\otimes \lambda}&\ll \mathbb{E}_\mu\left(S_{(m,n)}\right)+\intInd{[0,n-m]}{\intInd{[0,n-m]}{\frac{1}{x^{r(\mu)}}\frac{1}{(x+y)^p}\frac{1}{y^{pr(\mu)}}}{\lambda(y)}}{\lambda(x)}\\
&\ll \intInd{[0,\frac{\pi}{2}]}{\frac{1}{(\cos(\theta))^{r(\mu)}(\cos(\theta)+\sin(\theta))^p(\sin(\theta))^{pr(\mu)}}}{\lambda(\theta)}(n-m)^{2-r(\mu)-p(1+r(\mu))}.\end{array}\]
Let $\delta=2-r(\mu)-p(1+r(\mu))$.
Since  $p<\frac{r(\mu)}{1-r(\mu)}$ we have $\frac{\delta}{2}<1-p$.
Therefore, by {Sprindzuk}'s theorem,
\[S_N \overset{\mu\otimes \lambda-a.s}{=}\mathbb{E}(S_N)+O\left(N^{\frac{\delta}{2}}\log_2^{1+\varepsilon}\left(N\right)\right).\]
Hence  \[{\frac{S_N}{\mathbb{E}(S_N)}\xrightarrow[N\to+\infty]{\mu\otimes\lambda-a.s}1}.\]
\end{proof}
\begin{remark}
We observe that this "loss of power" in the {Borel-Cantelli} lemma above increases as the {Rajchman} order decreases.
\end{remark}

Next result relates the order of a {Rajchman} measure to Diophantine properties of its support and thus the SMTP property.
This result will therefore also show the Proposition \ref{diophraj}.

\begin{proposition}[General {Kurzweil} for Rajchman measures]

\label{kurzman}
Consider a Rajchman measure $\mu$ on $\mathbb{T}$ of order
 $0\le r(\mu)\le\frac{1}{2}$. 
Let $s>\frac{1}{r(\mu)}-1$. 
We have $\mu(\mathcal{D}io(s))=1$, or equivalently \[\mu-a.e\, \alpha \in [0,1[, \left(\mathbb{T},\lambda,T_\alpha\right) \,s-MSTP.\]
\end{proposition}

\begin{proof}
Let $C>0$, $j\in \mathbb{N}^*$ and $p\in \intervEnt{0}{j-1}$.
Consider \[\sum_{p=0}^{q-1}\mu\left(\left]\frac{p}{j}-\frac{C}{j^{s+1}},\frac{p}{j}+\frac{C}{j^{s+1}}\right[\right)=\intInd{\mathbb{T}}{\mathds{1}_{\left]\frac{p}{j}-\frac{C}{j^{s+1}},\frac{p}{j}+\frac{C}{j^{s+1}}\right[}}{\mu}=2\frac{C}{j^{s}}+\sum_{k\in \mathbb{Z}^*}\frac{\sin{\left(2\pi k \frac{C}{j^{s}}\right)}}{k\pi}\widehat{\mu}(kj).\]

So  \[\sum_{p=0}^{q-1}\mu\left(\left]\frac{p}{j}-\frac{C}{j^{s+1}},\frac{p}{j}+\frac{C}{j^{s+1}}\right[\right)\leq 2\frac{C}{j^{s+1}}+\frac{2C_\mu C^{\varepsilon}}{\pi j^{r(\mu)+s\varepsilon}}\zeta(1+r(\mu)-\varepsilon)\] with $\varepsilon\in ]0,r(\mu)[$.
And \[\mu\left(\union{p\in \intervEnt{0}{j-1}}{\left]\frac{p}{j}-\frac{C}{j^{s+1}},\frac{p}{j}+\frac{C}{j^{s+1}}\right[}\right)\leq 2\left(\frac{C}{j^s}+\frac{C_\mu C^\varepsilon}{\pi j^{r(\mu)+s\varepsilon}}\zeta(1+r(\mu)-\varepsilon)\right).\]
And we have by hypothesis that \[r(\mu)>\frac{1}{s+1}.\]
By strict inequality, we get that there exists $\varepsilon\in ]0,r(\mu)[$ such that \[r(\mu)>\frac{1}{s+1}\ and \  r(\mu)>1-s\varepsilon.\]
We also have the assumption that $r\leq \frac{1}{2}$.
By {Borel-Cantelli} \[\mu\left(\underset{n\to+\infty}{\overline{lim}}\left(\union{p\in \intervEnt{0}{n-1}}{\left]\frac{p}{n}-\frac{C}{n^{s+1}},\frac{p}{n}+\frac{C}{n^{s+1}}\right[}\right)\right)=0.\]
So for \[\mu-a.e\, \alpha \in [0,1[,\exists n\in \mathbb{N}^*, \forall k\geq n, d(k\alpha,\mathbb{Z})\geq \frac{C}{k^s}.\]
Let then take $\alpha\in [0,1[$ such that \[\exists n\in \mathbb{N}^*, \forall k\geq n, d(k\alpha,\mathbb{Z})\geq \frac{C}{k^s}.\]
This proves that $\alpha\in\mathcal{D}io(s)$, hence
$\mu(\mathcal{D}io(s))=1$. 
By {Kurzweil-Tseng} Theorem \ref{kurztseng}, \[{\left(\mathbb{T},\lambda,T_\alpha\right)\, s-MSTP}.\]
\end{proof}

Recall that when we have $r(\mu)>\frac{1}{2}$, the measure $\mu$ is absolutely continuous and therefore by Kurzweil theorem cited above $\mu-a.e\, \alpha\in [0,1[$, the system $\left(\mathbb{T},\lambda,T_{\alpha}\right)$ is $1-MSTP$.

\begin{remark}
We can note that the result obtained  in Proposition \ref{kurzman} remains consistent with that obtained with the Keplerian shear in the Theorem \ref{thm:dbcbyrp} because \[\sum_{k=1}^n \lambda\left(\left]z_k-\frac{C}{k^{\frac{1}{s}}},z_k+\frac{C}{k^{\frac{1}{s}}}\right[\right)^s=\sum_{k=1}^n \frac{2^sC^s}{k}\xrightarrow[n\to+\infty]{}+\infty.\]
\end{remark}


\bigbreak

\begin{remark}
For a {Rajchman} measure with positive order $r(\mu)>0$, $\mu-a.e\, \alpha\in [0,1[$ is Diophantine, that is, $Dio(\alpha)<+\infty$ $\mu- $almost surely.
\end{remark}

The following result shows that 
Proposition \ref{kurzman} is  optimal.
\begin{proposition}[Optimality of the Rajchman order \cite{Kauf}]
    Let $0<r< \frac12$.
    Then for all $1<s<\frac{1}{r}-1$, there exists $\mu$ a Rajchman measure such that $r(\mu)=r$ and $\mu(\mathcal{D}io(s))=0$.
\end{proposition}

\begin{proof}
Let $\varepsilon=\frac{1}{r}-(s+1)>0$.
We can consider $\mu_{\alpha}$ with $\alpha = s-1+\varepsilon$ constructed in the document of Kaufman \cite{Kauf} with its support in $E(\alpha)\subset [0,1[\setminus {\mathcal{D}io(s)}$.
\end{proof}

\begin{remark}
    With convex combination, for any $s>1$ we can construct a Rajchman measure $\mu$ for which $\mu(\mathcal{D}io(s))$ is equal to any value in $[0,1]$.
\end{remark}

\begin{exemple}[Diophantine property with self-similar measures]

The self-similar measure $\mu_{\theta}$ with $\theta>1$ not being {Pisot}, according to the work of {Pablo Shmerkin} and {Jean-Pierre Kahane} \cite{Kaha}, has an order of {Rajchman} $r\left(\mu_{\theta}\right)>0$.
In particular $\mu_{\theta}-a.e \,\alpha\in [0,1[$ is Diophantine.

\end{exemple}

\begin{exemple} [Liouville property with {Rajchman} measures on $S_\infty$]

The measure of {Rajchman} $\mu_{\infty}$ mentioned in \ref{cantraj} is {Rajchman} but since its support $S_\infty$ almost surely contains {Liouville} numbers, then we deduce by contrapositive of the Proposition \ref{kurzman} that $\mu_{\infty}$ does not have a strictly positive {Rajchman} order, i.e. that $r\left(\mu_{\infty}\right)=0$.

\end{exemple}

\section{Flow on compact Lie group bundle}

In this section we extend the results on Keplerian shear for tori bundles from Section~\ref{sec:main_cont} to non abelian bundles. The setting is essentially the same except that we replace the torus with a compact connected Lie group. It is worth mentioning that the core of the proofs will not differ too much. The fundamental reason being that
even if the group is not abelian, each individual orbit defines a one-parameter subgroup  which is therefore abelian, hence a torus.

\subsection{Lie group bundle and compatible flow}

Let $G$ be a compact connected {Lie} group. 
\begin{definition}
Let $(M,\mathcal{A})$ be a {Lindelöf} manifold of dimension $n\in\mathbb{N}^*$ of class $\mathcal{C}^1$.

Let $(\Omega,\mu)$ be a Borelian space, $G$ a connected compact {Lie} group and $\pi$ a continuous map from $\Omega$ into $M$.

$\Omega$ is a $G$-bundle if
\begin{enumerate}
\item{locally, we have for the charts $U$ of $\mathcal{A}$ a homeomorphism $\psi_U : \pi^{-1}(U)\to U\times G$}
\item{for all $U$ in $\mathcal{A},\pi_1\circ \psi_U=\pi_{\vert U}.$}
\end{enumerate}
\end{definition}

To define the continuous flow, we are going to use the exponential defined on the corresponding Lie algebra $Lie(G)$ of the group $G$.
Here the connectedness and the compactness of the group provide that the exponential $\exp\colon Lie(G)\to G$ is surjective, and injective from a neighborhood of the origin $B(0,\delta_{\mathrm{inj}})$  to a neighborhood of the neutral element of $e$ of $G$. The later implies that any continuous flow satisfying for all $y,z\in G, \phi_t(yz)=\phi_t(y)z$ is of the form $\phi_t(y)=\exp(tv)y$ for some $v\in Lie(G)$. This motivates the notion of compatible flow. 

\begin{definition}[Compatible flow]
Let $(g_t)_{t\in \mathbb{R}}$ a flow on $(\Omega,\mu)$.

$(g_t)_{t\in \mathbb{R}}$ is a compatible flow iff:

\begin{enumerate}
\item{There exists for all $U\in \mathcal{A}$ a measurable function $v_U : U\to Lie(G)$}
\item{$\forall U\in \mathcal{A},\forall (x,y)\in U\times G, \psi_U\circ g_t\circ\psi_U^{-1}(x,y)=(x,\exp(tv_U(x))y)$}.
\end{enumerate}
\end{definition}

\begin{definition}[Compatible measure]

Consider the notations previously established with the connected compact {Lie} group bundle.

Let $U\in \mathcal{A}$.
Let pose $\mu^U := \mu_{\psi_U}$.
And let pose $\mu^' := \mu_{\pi}$.
$\mu$ on $\Omega$ is compatible iff

\[\restreint{\mu^U}{\pi^{-1}(U)}=\restreint{\mu^'}{U}\otimes \mathcal{H}\,with\,\mathcal{H}\,the \,Haar\,measure\,on\,G.\]
\end{definition}


Next results by {Antoine Delzant} \cite{Adelz} will allow us to transfer the dynamics from $G$ to the torus.

\renewcommand{\theenumi}{\roman{enumi}}

\begin{proposition}
\label{prop_lie}    Let $G$ be a compact and connected Lie group.
\begin{enumerate}
    \item{Any compact, connected and abelian {Lie} group is isomorphic to a torus.}
    \item {\label{exist_max_torus}
    $G$ admits a maximal torus $T$, \[G=\union{g\in G}{gTg^{-1}}\] and for all maximal torus $T^'$, there exists $g\in G$ such that \[T^'=gTg^{-1}.\]}
    \item{\label{Lie_subgroup} Any closed subgroup $H$ of $G$ admits a group structure of {Lie} and $Lie(H)$ is a sev of $Lie(G)$}
    \item{\label{maxtor}

The orbital group of a direction $v\in Lie(G)$
\[ H_v :=\overline{\left\{\exp(tv)\in G : t\in \mathbb{R}\right\}}\]
is an abelian {Lie} group, compact and connected and so isomorphic to a torus.}
\item{\label{orbit_cover_max_torus}
    For all $v\in Lie(G)$, for all maximal torus $T$ of $G$, there exists $g\in G$ such that \[H_v\subset gTg^{-1}.\]}
\end{enumerate}
\end{proposition}

Now, we fix $T$ a maximal torus of $G$ and an isomorphism $\chi : T\to\mathbb{T}^R$ where $R$ is the dimension of $T$.

The combination of these results shows that each individual orbit can be embeded in the same torus $T$, a situation similar to the case of tori bundle, which is the setting of Section~\ref{sec:main_cont}. More precisely, by Proposition 
\ref{prop_lie}-\ref{orbit_cover_max_torus}, for $v\in Lie(G)$ we get a $g\in G$ such that $H_v\subset gTg^{-1}$.
To exploit this property we need a measurable selection of the element $g$, as a function of $v$.

Let $h : Lie(G)\to G$ be a measurable map\footnote{A construction of such a measurable function $h$ is done in Subsection~\ref{measurable_selection}, but other choices are possible.} such that for all $v\in Lie(G)$
\begin{equation}\label{inclusion_orbit}H_v\subset h(v)T(h(v))^{-1}.\end{equation}

Let pose for $U\in\mathcal{A}$ and $x\in U$, \begin{equation}\label{w_U_def}w_U(x)=D\chi(e)(Ad_{h(v_U(x))}(v_U(x)))\end{equation}
with $Ad$ the adjoint operator defined for $g\in G$ and $v\in Lie(G)$ by
\begin{equation}\label{def_Adjoint}Ad_g(v)=\left.\frac{d}{dt}g\exp(tv)g^{-1}\right\vert_{t=0}.\end{equation}

Let pose for nonzero $\xi\in\mathbb{Z}^d$ and $U\in\mathcal{A}$, $\nu^{(\xi,U)}:=\restreint{\mu^'}{U}_{\langle\xi\vert w_U(\cdot)\rangle}.$

\subsection{Keplerian shear and Rajchman measures on $G$-bundle}

\begin{theorem}
\label{Lieth} 
The dynamical system $(\Omega,\mu,(g_t)_{t\in\mathbb{R}})$ exhibits Keplerian shear iff for all nonzero $\xi\in \mathbb{Z}^R$, and for all $U\in\mathcal{A}$, $\restreint{\nu^{(\xi,U)}}{\mathbb{R}^*}$ is {Rajchman}
\end{theorem}

The plan of the demonstration will initially be to come back to a situation very similar to that of the torus bundle by relying on the maximal torus. To do this, we are going to use the orbital group $H_{v_U(x)}$ which is isomorphic to a torus.
The second step will be for a chart $U\in \mathcal{A}$ and $x\in U$ fixed to conjugate the action of $\exp(t v_U(x))$ on $G$ to a translation by $tw_U( x)$ on the torus.
Once immersed in this configuration provided by Lemma~\ref{lie2tore}, we will be able to carry out the same reasoning as in the torus in order to arrive at the property of {Rajchman} of $\restreint{\nu^{(\xi,U)}}{\mathbb{R}^*}$.

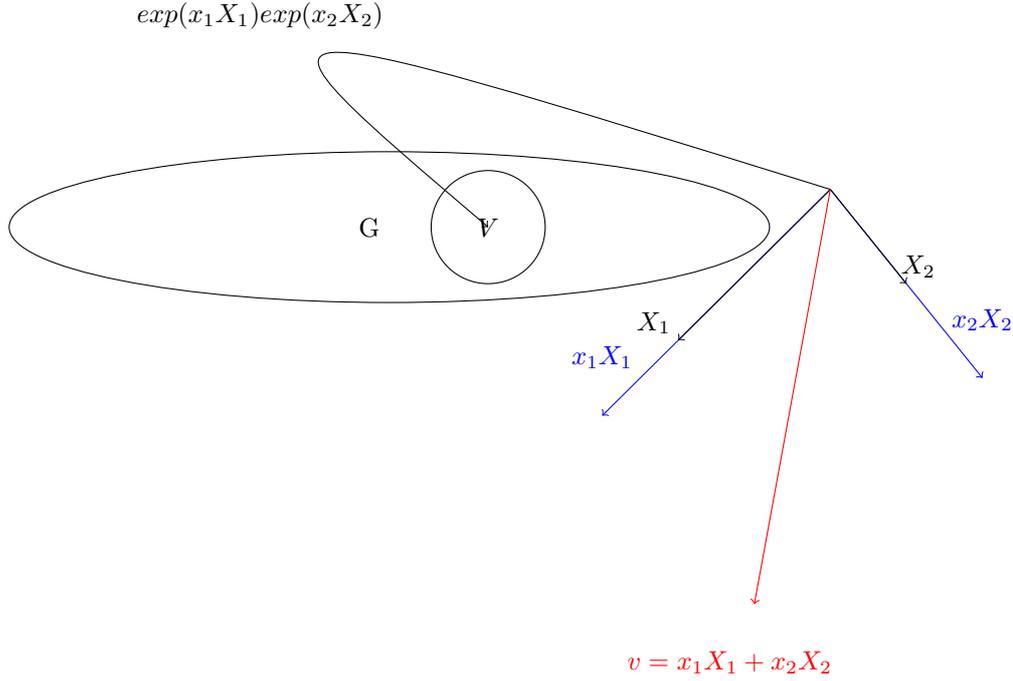
Consider the neighborhood $V$ of the local coordinate system of the second kind associated with $Lie(T)$ and $(Lie(T))^{\perp}$.

We cover $G$ with a finite family of elements $(g_k)_{k\in \intervEnt{1}{l}}\in G^{\intervEnt{1}{l}}$ s.t 
\begin{equation}\label{exploc}
G=\bigcup_{k=1}^lVg_k.
\end{equation}
Let $(W_k)_{k\in \intervEnt{1}{l}}$ s.t $W_1=Vg_1$ and \begin{equation}\label{partition}\forall k\in \intervEnt{2}{l},W_k=Vg_k\setminus \left(\union{j\in\intervEnt{1}{k-1}}{Vg_j}\right).\end{equation}

Let $\varphi :  G^2\to  G$ defined by $\varphi(x,y)=xy$ and $\varphi^v : G^2\to  G$ defined by $\varphi^v(x,y)=\varphi(h(v)x,y)$.
By \cite{NBour}, we can define 
 $$\begin{array}{lccl}
\psi_k: & Vg_k&\to & Lie(T)\times G\\
&x&\mapsto& \left(x_T,\exp(x_T^{\perp})g_k\right).
\end{array}$$

Let
$ \eta :  Lie(G)\times G\to G^2$ defined by 
$\eta(\theta,y)=\left(\exp(\theta),y\right)$.
We define for $v\in Lie(G),$ $k\in \intervEnt{1}{l}$ and $x\in Vg_k$,
\begin{equation}
\psi_k^v(x)=\psi_k((h(v))^{-1}x),
\end{equation}
\begin{equation}
    \varphi^v(x)=\varphi(h(v)x,y),
\end{equation}
for $(\theta,y)\in Lie(G)\times G$ and $t\in\mathbb{R}$,
\begin{equation}
    \phi_t^v(\theta,y)=(\theta+tAd_{h(v)}(v),y)
\end{equation}
and for $z\in G$,
\begin{equation}
    \chi_v(z)=\chi((h(v))^{-1}zh(v)).
\end{equation}
Thus, for all $k\in\intervEnt{1}{l}$,for all $x\in h(v)Vg_k$, we get
\begin{equation}\label{flot_translation}\left(\varphi^v\circ\eta\circ\phi_t^v\circ\psi_k^v\right)(x)=\exp(tv)x.\end{equation}

\begin{lemme}\label{lie2tore}
Let $U\in\mathcal{A}$ and $(f_1,f_2)\in \left(\mathbb{L}^2_{\restreint{\mu^'}{U}\otimes\mathcal{H}}\left(U\times G\right)\right)^2$.
We define for $t\in\mathbb{R}$ 
\[
b_t(f_1,f_2):=
\integraleMes{U\times G}{\overline{f_1}(x,\exp\left(tv_U(x)\right)y)f_2(x,y)}{\restreint{\mu^'}{U}\otimes\mathcal{H}(x,y)}.
\]

Then \begin{equation}\label{bt}
b_t(f_1,f_2)=\sommeInd{k\in\intervEnt{1}{l}}{\integraleMes{U}{\left(\integraleMes{\mathbb{T}^{R}\times h(v_U(x))W_k}{\overline{\check{f_1}}(x,z+tw_U(x),r))\check{f_2}(x,z,r))}{m^'_{(x,k)}}(z,r)\right)}{\restreint{\mu^'}{U}(x)}}
\end{equation}
with 
$$
\check{f_1} : (x,z,r)\mapsto f_1(x,\varphi^{v_U(x)}(\chi_{v_U(x)}^{-1}\left(z\right),r)) \text{ and } \check{f_2} : (x,z,r)\mapsto f_2(x,\varphi^{v_U(x)}(\chi_{v_U(x)}^{-1}\left(z\right),r)).
$$
and $m^'_{(x,k)}=\left(\chi_{v_U(x)}\circ\exp\right)_{*}m_{(x,k)}$ is the push forward measure obtained by transfer theorem, where $m_{(x,k)}=\left(\psi_k^{v_U(x)}\right)_{*} \mathcal{H}$ is the push forward measure associated.
We recall that $D\chi_{v_U(x)}(e)=D\chi(e)\circ Ad_{h(v_U(x)).}$
\end{lemme}

\begin{proof}

Let non-zero $\xi\in \mathbb{Z}^R$.
Let $U\in\mathcal{A}$ and $(f_1,f_2)\in \left(\mathbb{L}^2_{\restreint{\mu^'}{U}\otimes\mathcal{H}}\left(U\times G\right)\right)^2$.
We have the following expansion
\begin{align*}
b_t(f_1,f_2)&:=
\integraleMes{U\times G}{\overline{f_1}(x,\exp\left(tv_U(x)\right)y)f_2(x,y)}{\restreint{\mu^'}{U}\otimes\mathcal{H}(x,y)}\\
&=\integraleMes{U}{\integraleMes{G}{\overline{f_1}(x,\exp\left(tv_U(x)\right)y)f_2(x,y)}{\mathcal{H}(y)}}{\restreint{\mu^'}{U}(x)}\\
&=\sum_{k=1}^l\integraleMes{U}{\left(\integraleMes{h(v_U(x))W_k}{\overline{f_1}(x,\check{\varphi^{v_U(x)}}\circ\phi_t^{v_U(x)}\circ\psi_k^{v_U(x)}(y))f_2(x,\check{\varphi^{v_U(x)}}\circ\psi_k^{v_U(x)}(y))}{\mathcal{H}(y)}\right)}{\restreint{\mu^'}{U}(x)}\\
&=\sum_{k=1}^l\integraleMes{U}{\left(\integraleMes{Lie(T)\times h(v_U(x))W_k}{\overline{f_1}(x,\check{\varphi^{v_U(x)}}\circ\phi_t^{v_U(x)}(\theta,r))f_2(x,\check{\varphi^{v_U(x)}}(\theta,r))}{m_{(x,k)}}(\theta,r)\right)}{\restreint{\mu^'}{U}(x)}.
\end{align*}
We have completed the first step, by placing ourselves in the {Lie} algebra of the maximal torus $T$ which is abelian, so
\[
\check{\varphi}\circ\phi_t^{v_U(x)}(\theta,r)=\check{\varphi}(\theta+tv_U(x),r)=\varphi(\exp(\theta+tv_U(x)),r)=\varphi(\exp(tv_U(x))\exp(\theta),r)
.\]
We are now at the second step placing ourselves in the sections of the Lie group G by $T$.
So 
\begin{align*}
&\quad b_t(f_1,f_2)\\
&=\sum_{k=1}^l\integraleMes{U}{\left(\integraleMes{Lie(T)\times h(v_U(x))W_k}{\overline{f_1}(x,\varphi(\exp(tv_U(x))\chi_{v_U(x)}^{-1}\left(\chi_{v_U(x)}\left(e^\theta)\right)\right),r))f_2(x,\varphi(e^\theta),r))}{m_{(x,k)}}(\theta,r)\right)}{\restreint{\mu^'}{U}(x)}\\
&=\sum_{k=1}^l\integraleMes{U}{\left(\integraleMes{\mathbb{T}^{R}\times h(v_U(x))W_k}{\overline{f_1}(x,\varphi(\exp(tv_U(x))\chi_{v_U(x)}^{-1}\left(z\right),r))f_2(x,\varphi(\chi_{v_U(x)}^{-1}\left(z\right),r))}{m^'_{(x,k)}}(z,r)\right)}{\restreint{\mu^'}{U}(x).}
\end{align*}
We pass here in the torus isomorphic to $T$. Since
\[\exp(tv_U(x))\chi_{v_U(x)}^{-1}\left(z\right)=\chi_{v_U(x)}^{-1}(\chi_{v_U(x)}(\exp(tv_U(x)))+z)\]
and \[\chi_{v_U(x)}(\exp(tv_U(x)))+z=z+tD\chi_{v_U(x)}(e)=z+tw_U(x),\]
we get
\begin{align*}
&\quad b_t(f_1,f_2)\\
&=\sum_{k=1}^l\integraleMes{U}{\left(\integraleMes{\mathbb{T}^{R}\times h(v_U(x))W_k}{\overline{f_1}(x,\varphi(\chi_{v_U(x)}^{-1}\left(z+tw_U(x)\right),r))f_2(x,\varphi(\chi_{v_U(x)}^{-1}\left(z\right),r))}{m^'_{(x,k)}}(\theta,r)\right)}{\restreint{\mu^'}{U}(x).}
\end{align*}
This proves \eqref{bt}.
\end{proof}
We are then in the case of a torus foliation by the variable $x$ which will then allow us
to apply the same reasoning as in the case of the torus.
\begin{proof}[Proof of Theorem \ref{Lieth}]
Let begin with the direct implication.
Suppose that the dynamical system $(\Omega,\mu,(g_t)_{t\in\mathbb{R}})$ exhibits Keplerian shear.
Let $U\in\mathcal{A}$ and let \[\xi\in\mathbb{Z}^R\setminus\{0\}.\]
Let $(f_1,f_2)\in \left(\mathbb{L}^2_{\restreint{\mu^'}{U}\otimes\mathcal{H}}\left(U\times G\right)\right)^2$ such that :
\[\check{f_1} : (x,z,r) \mapsto e^{2i\pi\langle\xi\vert z\rangle}\]
and \[\check{f_2} : (x,z,r) \mapsto \mathds{1}_{\langle\xi\vert w_U(\cdot)\rangle\ne  0}(x)e^{2i\pi\langle\xi\vert z\rangle}\] taking notations of Lemma \ref{lie2tore}.
Then
\[\integraleMes{\mathbb{R}}{e^{-2i\pi tz}}{\nu^{\xi,U}(z)}=\integraleMes{U}{e^{-2i\pi t\langle\xi\vert w_U(x)\rangle}g(x)}{\mu^'(x)}\]
with
\[g : x\in U\mapsto \mathds{1}_{\langle\xi\vert w_U(\cdot)\rangle\ne 0}(x).\]
Thus, by Lemma \ref{lie2tore}
\[\integraleMes{U}{e^{-2i\pi t\langle\xi\vert w_U(x)\rangle}g(x)}{\restreint{\mu^'}{U}(x)}=\integraleMes{U\times G}{\overline{f_1}(x,\exp\left(tv_U(x)\right)y)f_2(x,y)}{\restreint{\mu^'}{U}\otimes\mathcal{H}(x,y)}.\]
By the Keplerian shear property the right hand term converges and as in (\ref{birkth}), the limit is zero.
So \[{\integraleMes{\mathbb{R}}{e^{-2i\pi tz}}{\nu^{\xi,U}(z)}\xrightarrow[t\to\pm\infty]{}0}.\]

Let us continue with the reciprocal implication.
Let $R$ be the dimension of the maximal torus of $G$.

Consider  $k\in\intervEnt{1}{l}$ and define the measure $\omega^k$ on $U\times \mathbb{T}^R\times W_k$  by
\[\forall A\in \mathcal{B}(U\times \mathbb{T}^R\times h(v_U(x))W_k), \omega^k(A)=\integraleMes{U}{\integraleMes{\mathbb{T}^R\times h(v_U(x))W_k}{\mathds{1}_A(x,y,z)}{m^'_{(x,k)}(y,z)}}{\mu^'_{\vert U}(x)}.\]
Let $U\in\mathcal{A}$ and $(\xi_1,\xi_2)\in (\mathbb{Z}^R)^2$. Let $(f_1,f_2)\in \left(\mathbb{L}^2_{\restreint{\mu^'}{U}\otimes\mathcal{H}}\left(U\times G\right)\right)^2$ such that,
\[\check{f_1} : (x,z,r) \mapsto a_1(x,r)e^{2i\pi\langle\xi_1\vert z\rangle}\]
and \[\check{f_2} : (x,z,r) \mapsto a_2(x,r)e^{2i\pi\langle\xi_2\vert z\rangle}\]
with $a_1$ and $a_2$ square summable functions in the appropriate {Lebesgue} space.
If $m\neq m^{'}$, the scalar product between $f_1,f_2$ is zero. Therefore we assume that 
$m=m^{'}$.
So, by Lemma \ref{lie2tore}  
\begin{equation}\label{expfunction}\integraleMes{U\times G}{\overline{f_1}(x,\exp\left(tv_U(x)\right)y)f_2(x,y)}{\restreint{\mu^'}{U}\otimes\mathcal{H}(x,y)}=\integraleMes{U}{e^{-2i\pi t\langle\xi_1\vert w_U(x)\rangle}g(x)}{\restreint{\mu^'}{U}(x)}\end{equation}
with
\[g : x\in U\mapsto \sum_{k=1}^l\left(\integraleMes{h(v_U(x))W_k}{\overline{a_1}(x,r)a_2(x,r)e^{2i\pi\langle\xi_2-\xi_1\vert z \rangle}}{m^'_{(x,k)}(z,r)}\right).\]
By hypothesis, when $\xi_1\ne 0$, \[\integraleMes{\mathbb{R}}{e^{-2i\pi tz}}{\nu^{\xi_1,U}(z)}\xrightarrow[t\to\pm\infty]{}0.\]
And then, by Lemma \ref{convfaible}, we have that $\exp(2i\pi t\cdot)$ converge weakly-$*$ in $\mathbb{L}^{\infty}_{\nu^{\xi_1,U}}(\mathbb{R})$ to $0$.
Thus, \[\integraleMes{U}{e^{-2i\pi t\langle\xi_1\vert w_U(x)\rangle}g(x)}{\restreint{\mu^'}{U}(x)}\xrightarrow[t\to\pm\infty]{}\intInd{\{\langle\xi_1\vert w_U(\cdot)\rangle=0\}\times G}{\overline{f_1}(x,y)f_2(x,y)}{\mu^'_{\vert U}\otimes \mathcal{H}(x,y)}.\]
Then, remembering (\ref{expfunction}), we get 
\[
\integraleMes{U\times G}{\overline{f_1}(x,\exp\left(tv_U(x)\right)y)f_2(x,y)}{\restreint{\mu^'}{U}\otimes\mathcal{H}(x,y)}\xrightarrow[t\to\pm\infty]{}\intInd{\{\langle\xi_1\vert w_U(\cdot)\rangle=0\}\times G}{\overline{f_1}(x,y)f_2(x,y)}{\mu^'_{\vert U}\otimes \mathcal{H}(x,y)}.
\]
By totality obtained by the {Fourier} series expansion, this gives that for all functions $f_1,f_2\in\mathbb{L}^2_{\mu^'_{\vert U}\otimes\mathcal{H}}(U\times G)$
\begin{equation}\label{denselie}
\integraleMes{U\times G}{\overline{f_1}(x,\exp\left(tv_U(x)\right)y)f_2(x,y)}{\restreint{\mu^'}{U}\otimes\mathcal{H}(x,y)}\xrightarrow[t\to\pm\infty]{}\intInd{U\times G}{\overline{E_{\mu^'_{\vert U}\otimes\mathcal{H}}(f_1\vert \mathcal{I}_U)}E_{\mu^'_{\vert U}\otimes\mathcal{H}}(f_2\vert \mathcal{I}_U)}{\mu^'_{\vert U}\otimes\mathcal{H}}.
\end{equation}
\end{proof}


\subsection{Keplerian shear on $G$-bundle in the regular case}

Now, the aim here is to give an analogue of the fundamental Theorem 3.3 in \cite{DaTho}, which guarantees Keplerian shear for flows with regular velocities and the negligible critical points, for absolutely continuous measures.

\begin{theorem}\label{Lie_Theorem_kepler_regular}
If for all $U\in \mathcal{A}$, $w_U$ is $\mathcal{C}^1$, 
$\mu^'_{\vert U}\ll\lambda$ and
$$\mu\left(\union{\xi\in\prive{\mathbb{Z}^d}{\left\{0\right\}}}{\left\{x\in U : d\langle\xi\vert w_U(x)\rangle=0\right\}\setminus\{x\in U : \langle\xi\vert w_U(x)\rangle=0\}}\right)=0,$$
then the dynamical system $(\Omega,\mu,(g_t)_{t\in\mathbb{R}})$ has Keplerian shear.
\end{theorem}

\begin{lemme}\label{critiques_negli}
Let $V\subset \mathbb{R}^d$ be an open set and $f\colon V\to\mathbb{R}$ a $\mathcal{C}^1$ map. If $\{d_xf=0\}\setminus \{f=0\}$ has zero Lebesgue measure then for any $a\in\mathbb{R}^*$, $\{f=a\}$ has zero Lebesgue measure.
\end{lemme}
\begin{proof}
Let $(e_i)$ be the canonical basis of $\mathbb{R}^d$ and denote by $\mathcal{C}_i$ the cone of vectors $h\neq0$ such that the angle between $h$ and $e_i$ is less than or equal to $\pi/8$. 

Let $a\in\mathbb{R}^*$ and set $A=\{f=a\}$ and $A^'$ the set of Lebesgue density points of $A$, that is $\lambda(A\cap B(x,r))/\lambda(B(x,r))\to 1$ as $r\to0$ for each $x\in A^'$. 

Suppose that $x\in A^'$.
For each $i=1,\ldots,d$ there exists a sequence $h_n^i\in\mathcal{C}_i$ such that $x+h_n^i\in A$ and $h_n^i\to0$. Without loss of generality we may assume that $h_n^i/\|h_n^i\|$ converges to some $h^i\in\mathcal{C}_i$. Since $f(x+h_n^i)=a=f(x)$, a first order expansion shows that $d_xf(h^i)=0$. Since the $h_i$ form a basis of $\mathbb{R}^d$, this implies that $d_xf=0$. 

By assumption we get that $A^'$, and therefore $A$, has zero measure.
\end{proof}

\begin{proof} [Proof of Theorem \ref{Lie_Theorem_kepler_regular}]
Let $U\in\mathcal{A}$.
Let us show that the measures $\nu^{(\xi,U)}$ have the {Rajchman} property.
Let $\xi\in\prive{\mathbb{Z}^d}{\{0\}}$.
By {Radon-Nikodym}
\[\integraleMes{U}{e^{-2i\pi t\langle\xi\vert w_U(x)\rangle}}{\restreint{\mu^'}{U}(x)}=\integraleMes{U}{e^{-2i\pi t\langle\xi\vert w_U(x)\rangle}\frac{d\mu^'_{\vert U}}{d\lambda}(x)}{\lambda(x)}.\]
But,
\[\mu\left(\union{\xi\in\prive{\mathbb{Z}^d}{\left\{0\right\}}}{\left\{x\in U : d\langle\xi\vert w_U(x)\rangle=0\right\}\setminus\{x\in U : \langle\xi\vert w_U(x)\rangle=0\}}\right)=0.\]
Thus, by Lemma \ref{critiques_negli},  \[\forall a\in\mathbb{R}^*,\mu\left(\union{\xi\in\prive{\mathbb{Z}^d}{\left\{0\right\}}}{\left\{x\in U : \langle\xi\vert w_U(x)\rangle=a\right\}}\right)=0.\]

By local normal form of submersions, we can choose a countable family $(\varphi_k, V_k)_{k\in \mathbb{N}}$ of open charts on $M$ which are pariwise disjoints, cover $U\setminus\left(\union{\xi\in\prive{\mathbb{Z}^d}{\left\{0\right\}}}{\left\{x\in U : d\langle\xi\vert w_U(x)\rangle=0\right\}}\right)$ up to a Lebesgue-negligible set and for all $k\in\mathbb{N}$ and all $x\in \varphi_k(V_k)$, \[\langle \xi\vert w_U(\varphi^{-1}_k(x))\rangle=x_1.\]
Therefore, for $k\in\mathbb{N}$,
\[\intInd{V_k}{e^{-2i\pi t\langle\xi\vert w_U(x)\rangle}\frac{d\mu^'_{\vert U}}{d\lambda}(x)}{\lambda(x)}=\intInd{\varphi_k(V_k)}{e^{-2i\pi tz}\frac{d\mu^'_{\vert U}}{d\lambda}(\varphi_k^{-1}(z))\abso{\det(J_{\varphi_k^{-1}})(z)}}{\lambda(z)}.\]
In the other hand, we have 
\[\integraleMes{U}{e^{-2i\pi t\langle\xi\vert w_U(x)\rangle}\frac{d\mu^'_{\vert U}}{d\lambda}(x)}{\lambda(x)}=\intInd{\langle\xi\vert w_U\rangle=0}{\frac{d\mu^'_{\vert U}}{d\lambda}(x)}{\lambda(x)}+\sum_{k\in\mathbb{N}}\intInd{V_k}{e^{-2i\pi t\langle\xi\vert w_U(x)\rangle}\frac{d\mu^'_{\vert U}}{d\lambda}(x)}{\lambda(x)}.\]
Thus, by {Riemann-Lebesgue}, \[\integraleMes{U}{e^{-2i\pi t\langle\xi\vert w_U(x)\rangle}\frac{d\mu^'_{\vert U}}{d\lambda}(x)}{\lambda(x)}\xrightarrow[t\to+\infty]{}\intInd{\langle\xi\vert w_U\rangle=0}{\frac{d\mu^'_{\vert U}}{d\lambda}(x)}{\lambda(x)}.\]
Therefore, \[\integraleMes{U}{e^{-2i\pi t\langle\xi\vert w_U(x)\rangle}}{\restreint{\mu^'}{U}(x)}\xrightarrow[t\to+\infty]{}\mu^'_{\vert U}(\{\langle\xi\vert w_U\rangle=0\}).\]
By Theorem \ref{Lieth}, we get that $(\Omega,\mu,(g_t)_{t\in\mathbb{R}})$ exhibits Keplerian shear.
\end{proof}
%


We can develop a corollary of Theorem \ref{Lie_Theorem_kepler_regular}. We define \begin{equation}
    \label{spheric_constant_g}
    \bar h : v\in Lie(G)\setminus\{0\}\mapsto h\left(\frac{1}{\Vert v\Vert_2 }v\right)
\end{equation} and we assume that $\bar h$ is $\mathcal{C}^1$ on a full measure open set.

\begin{coro}\label{coro_regular}
    We suppose that for a countable family $(U_j)_{j\in\mathbb{N}}$ of charts which cover $M$ up to a Lebesgue-negligible set, $\bar h$ is $\mathcal{C}^1$ on $V_j$ such that $v_{U_j}(U_j)\subset V_j$, $\mu^'_{\vert U_j}\ll \lambda$, $v_{U_j}$ is a $\mathcal{C}^1$-subimmersion and for $\mu^'_{\vert U_j}-a.a\ x\in U_j$,  $Rank(Dv_{U_j}(x))\geq 1$.
    
    We get that the dynamical system $(\Omega,\mu,(g_t)_{t\in\mathbb{R}})$ exhibits Keplerian shear.
\end{coro}

First, for $R\in \mathbb{N}^*$, $0\neq\xi\in\mathbb{Z}^R$ we define the function \begin{equation}\label{scalar_def}
    S_\xi : v\in Lie(G)\mapsto \langle\xi\vert D\chi(e)(Ad_{h(v)}(v))\rangle .
\end{equation}
We'll need for $v\in Lie(G)$ the differential of $A^v : g\in G\mapsto Ad_g(v)$.

\begin{lemme}\label{adjoint_diff}
For $v,w\in Lie(G)$, noting $[\cdot,\cdot]$ the Lie bracket, we have
    \[DA^v(g)(w)=[Ad_g(w),Ad_g(v)].\]
\end{lemme}

By \eqref{inclusion_orbit}, for all $v\in Lie(G)\setminus\{0\}$ and all $t\in\mathbb{R}$,\begin{equation}
    \label{stability_C}(\bar h(v))^{-1}\exp(tv)\bar h(v)\in T.
\end{equation}

To ensure the regularity of $S_\xi$, we need next lemmas.
\begin{lemme}\label{constance_C}
    For all $v\in Lie(G)\setminus\{0\}$, \[D\bar h(v)(v)=0.\]
\end{lemme}

\begin{proof}
    Let $t\in ]-1,1[$ and $v\in Lie(G)\setminus\{0\}$.
    Thus $\bar h(v+tv)=h\left(\frac{1+t}{(1+t)\Vert v\Vert}v\right)=\bar h(v).$
    Therefore, \[D\bar h(v)(v)=0.\]
\end{proof}

\begin{lemme}\label{singularity_S_xi}
    
    Let $V\subset Lie(G)\setminus\{0\}$ and $W\subset \mathbb{R}^d$ open sets.
    We suppose that the restriction of $h$ on $V$ is $\mathcal{C}^1$.
    Let $\psi : V\to W$ be a $\mathcal{C}^1$-diffeomorphism.
    We have for $w\in W$,
    \[S_\xi(\psi^{-1}(w))\ne 0\implies D(S_\xi\circ \psi^{-1})(w)\left(\left(D\psi^{-1}(w)\right)^{-1}\circ \psi^{-1}(w)\right)\ne 0.\]
\end{lemme}
\begin{proof}
    Let $w\in W$ such that $S_\xi(\psi^{-1}(w))\ne 0$.
    By Lemma \ref{adjoint_diff}, for $u\in\mathbb{R}^d$ we get
    \[D(S_\xi\circ \psi^{-1})(w)(u)=\left(\begin{array}{ll}\langle \xi\vert D\chi(e)(Ad_{\bar h(\psi^{-1}(w))}\circ D\psi^{-1}(w)(u))\rangle\\+\langle\xi\vert D\chi(e)\circ Ad_{\bar h (\psi^{-1}(w))}\circ[D\bar h({\psi^{-1}(w)})\circ D\psi^{-1}(w)(u),\psi^{-1}(w)]\rangle.\end{array}\right)\]
    Taking $u=\left(D\psi^{-1}(w)\right)^{-1}\circ \psi^{-1}(w)$, by Lemma \ref{constance_C},\[D(S_\xi\circ \psi^{-1})(w)(u)=S_\xi(\psi^{-1}(w))\ne 0.\]
\end{proof}

This intermediate result helps us to connect Lemmas \ref{singularity_S_xi} and \ref{regular_w_U}.

\begin{lemme}\label{rank_regular_v}
    Let $U\subset M$ and $x\in U$.
    Let $v : U\to V\subset Lie(G)$ a $\mathcal{C}^1$ function and subimmersion around $x$ such that $k:=rank(Dv(x))\geq 1$.
    We have that there exists a $\mathcal{C}^1$-diffeomorphism $\psi : V\to V^'\subset Lie(G)$ and a neighbourhood $W\subset U$ of $x$ such that for all $a\in W$ \[\psi(v(a))\in Im(D(\psi\circ v)(a)).\]
\end{lemme}
\begin{proof}
Let $U\subset M$ and $x\in U$.
Let $v : U\to V\subset Lie(G)$ a $\mathcal{C}^1$ function and subimmersion around $x$ such that $k=rank(Dv(x))\geq 1$.
Using local normal form of subimmersions, there exists $V^'\subset \mathbb{R}^d$, $U^'$ open sets and $\mathcal{C}^1$-diffeomorphisms  $\psi : V\to V^'$ and $\varphi : U\to U^'$ such that for all $z\in U^'$, \[\psi(v(\varphi^{-1}(z)))=(z_1,\cdots,z_k,0,\cdots,0).\]
Let $a\in U$.
Therefore, \[D(\psi\circ v)(a)(u)=((D\varphi(a)(u))_1,\cdots,(D\varphi(a)(u))_k,0,\cdots,0).\]
Taking $u=(D\varphi(a))^{-1}(\varphi(a))$, we get
\[D(\psi\circ v)(a)(u)=\psi(v(a)).\]
\end{proof}

We define here $w_U(x)=D\chi(e)(Ad_{\bar h(v_U(x))}(v_U(x)))$, where we recall that $\bar h$ is defined in \eqref{spheric_constant_g}. We emphasize that changing $h$ to $\bar h$ does not affect the role of $w_U$, in particular Theorem \ref{Lie_Theorem_kepler_regular} applies verbatim.
We finally get that $w_U$ is regular in the following sense.

\begin{lemme}\label{regular_w_U}
    Suppose that $v_U$ is a $\mathcal{C}^1$ subimmersion with a rank greater than $1$ $\mu-$almost everywhere.
    We get that $w_U$ is $\mathcal{C}^1$ on a full measure open set such that
    \[\mu\left(\union{\xi\in\prive{\mathbb{Z}^d}{\left\{0\right\}}}{\left\{x\in U : d\langle\xi\vert w_U(x)\rangle=0\right\}\setminus\{x\in U : \langle\xi\vert w_U(x)\rangle=0\}}\right)=0.\]
\end{lemme}
\begin{proof}
    By hypothesis, $v_U$ is subimmersion with a rank greater than $1$ $\mu-$almost everywhere.
    Thus, by Lemma \ref{rank_regular_v}, for $\mu-$almost all $x\in U$, there exists a neighbourhood $W$ of $x$, $V$ an open neighbourhood of $v_U(x)$, $\psi : V\to V^'\subset Lie(G)$ a $\mathcal{C}^1$-diffeomorphism  such that for all $a\in W$, \[\psi(v_U(a))\in Im(D(\psi\circ v_U)(a)).\]
    Let $x\in U$ such that $\langle\xi\vert w_U(x)\rangle\ne 0$.
    Therefore, \[\langle\xi\vert D\chi(e)(Ad_{\bar h(v_U(x))}(\psi^{-1}(\psi(v_U(x)))))\rangle\ne 0.\]
    We have\[d\langle\xi\vert w_U(x)\rangle(u)=D(S_\xi\circ \psi^{-1})(\psi\circ v_U(x))(D(\psi\circ v_U)(x)(u)).\]
    By Lemma \ref{rank_regular_v}, we can get $u$ such that \[d\langle\xi\vert w_U(x)\rangle(u)=D(S_\xi\circ \psi^{-1})(\psi\circ v_U(x))(v_U(x))=\langle\xi\vert w_U(x)\rangle\ne 0.\]
    Thus, \[\mu\left(\union{\xi\in\prive{\mathbb{Z}^d}{\left\{0\right\}}}{\left\{x\in U : d\langle\xi\vert w_U(x)\rangle=0\right\}\setminus\{x\in U : \langle\xi\vert w_U(x)\rangle=0\}}\right)=0.\]
\end{proof}

\begin{proof} [Proof of Corollary \ref{coro_regular}]
    By Lemma \ref{regular_w_U}, we get that $w_{U_j}$ is $\mathcal{C}^1$ and \[\mu\left(\union{\xi\in\prive{\mathbb{Z}^d}{\left\{0\right\}}}{\left\{x\in U_j : d\langle\xi\vert w_{U_j}(x)\rangle=0\right\}\setminus\{x\in U_j : \langle\xi\vert w_{U_j}(x)\rangle=0\}}\right)=0.\]
    Applying Theorem \ref{Lie_Theorem_kepler_regular}, we get the result.
\end{proof}

\subsection{Measurable selection from the orbital group to the maximal torus}\label{measurable_selection}

We detail here how to build a measurable function $h$ which fullfils all the requirements \eqref{inclusion_orbit}. First we ensure the existence of a measurable selection of the element $g$ which conjugates an element of $G$ to another one in the fixed maximal torus $T$.
To select this $g$, we will use the Kuratowski theorem proved in \cite{Parth}.

\begin{theorem}[Kuratowski]\label{Kuratowski}
    Let $(X,d)$ a Polish space and $(\Omega,\mathcal{F})$ a measurable space.
    Let $F$ a multifonction from $\Omega$ to $X$ such that for all $\omega\in \Omega$ and for all open set $U$ of $X$,\[F(\omega)=\overline{F(\omega)}\ and\ F^{-1}(U)\in\mathcal{F}.\]
    There exists a $\mathcal{F}-\mathcal{B}(X)$ measurable function $f : \Omega\to X$ such that for all $\omega\in \Omega$,\[f(\omega)\in F(\omega).\]
\end{theorem}

Now, we build the multifunction $F$ which allows us to identify a measurable function $f$ such that for all $z\in G$, $f(z)z(f(z))^{-1}\in T$.
Let \[F : z\in G\mapsto \{g\in G : z\in gTg^{-1}\}.\]
Now, we need to have $F(z)$ closed and non-empty.

\begin{lemme}\label{closed_F}
    For all $z\in G$, $F(z)\ne \emptyset$ is closed
\end{lemme}

\begin{proof}
    Let $z\in G$.
    By Proposition \ref{prop_lie}-\ref{exist_max_torus}, we get that $F(z)\ne \emptyset$.
    Let $g\in \overline{F(z)}$.
    There exists $(g_n)_{n\in\mathbb{N}}\in (F(z))^\mathbb{N}$ such that \[g_n\xrightarrow[n\to+\infty]{}g.\]
    Thus, for all $n$, $g_n^{-1}zg_n\in T$.
    We know that $T$ is closed and $G$ is a Lie group, thus, \[g^{-1}zg\in T.\]
\end{proof}

Now, we have to prove that for all open set $U$ of $G$, $F^{-1}(U)$ is a measurable set.

\begin{lemme}\label{measurable_F}
    For all open set $U$, $F^{-1}(U)$ is a measurable set.
\end{lemme}

\begin{proof}
    Let \[P : (g,t)\in G\times T\mapsto gtg^{-1}.\]
    $P$ is continuous because $G$ is a Lie group.
    Let $U$ an open set of $G$.
    We see immediately that \[F^{-1}(U)=P(U\times T).\]
    We know that $G$ is a compact connected Lie group, thus $G$ is metrizable and $U$ is a countable union of compact sets.
    By compactness of $T$ and continuity of $P$, $P(U\times T)$ is a countable union of compact set.
    Then, $F^{-1}(U)$ is a measurable set.
\end{proof}

Now, by Lemma \ref{closed_F}, \ref{measurable_F} and Theorem \ref{Kuratowski}, there exists a measurable function $f : G\to G$ such that for all $z\in G$, \begin{equation}\label{measurable_function_conj}f(z)z(f(z))^{-1}\in T.\end{equation}

With the flow orbit, we can make a semi-fibration on Lie group to use torus property as in the previous theorem in the tori bundle case.
We note that for all $v\in Lie(G)$, $v_{\mathrm{inj}}:=\frac{\delta_{\mathrm{inj}}}{\norme{v}_2+1}v\in B(0,\delta_\mathrm{inj})$.
Let $h(v)=f\left(\exp(v_\mathrm{inj})\right)$.
Writing $T^'=h(v)^{-1}T h(v)$,
by injectivity of $\exp$ on $B(0,\delta_\mathrm{inj})$ and the definition of $f$, $v_{\mathrm{inj}}\in Lie(T^')$. Thus, for all $t\in\mathbb{R}$, $tv\in Lie(T^')$.
Taking the exponential shows that $H_v\subset T^'$, 
since $T^'$ is closed.
Thus, \eqref{inclusion_orbit} is satisfied by $h$.
\subsection{Main examples of compact Lie groups where the results can be applied}

\begin{exemple}[Torus]

The simplest example of a compact Lie group is a torus $\mathbb{T}^d$, which is abelian. This is the framework of the previous sections. In particular with $\mathbb{T}^2$, a measurable velocity vector $v(x)$, and for a single fiber to torus, the flow:
\[g_t : (x,y)\in \mathbb{T}^2\mapsto (x,y+tv(x))\]
\end{exemple}

\begin{exemple}[Spinorial groups]

The spin group for $n\geq 2$, $Spin(n)$ is a compact and connected {Lie} group, which allows us to use it as an example.
\end{exemple}
\begin{exemple}[Orthogonal groups]

The orthogonal groups for $n\geq 2$ $SO_n(\mathbb{R})$ are connected compact {Lie} groups, non abelian if $n\ge3$.
We can consider the flow $g_t : M\in SO_3(\mathbb{R})\mapsto exp(tA)M$ with $A$ an antisymmetric matrix.
\end{exemple}

\begin{exemple}[Special unitary group]

A special unitary group for $n\geq 2$, $SU_n(\mathbb{R})$ is also a connected compact {Lie} group and so we can use the previous theorems. Moreover, these {Lie} groups are even simply connected. They are always non-abelian groups.
The simplest example is $SU_2(\mathbb{R})$ which is isomorphic to the hypersphere $\mathbb{S}^3$ of $\mathbb{R}^4$.
The flow \[g_t : M\in SU_2(\mathbb{R}) \mapsto \exp(tA)M\] is well defined for $A\in \mathcal{M}_2(\mathbb{R})$ such that $^t A=-A^*$.
\end{exemple}


\begin{thebibliography}{2} 
   \bibitem[AGZV]{SiDM} V.I. Arnold, S.M. Gusein-Zade and A.N. Varchenko, \emph{Singularities of differential Maps Volume 2: Monodromy and Asymptotics of Integrals}, Birkhäuser, 1988 

\bibitem[Ath]{Atry} Jayadev Athreya, Logarithm laws and shrinking target properties, \emph{Proc. Indian Acad. Sci. (Math. Sci.)} Vol. 119, No. 4, September 2009, pp. 541–557.

\bibitem[Blu]{CBluhm} Christian Bluhm,
Liouville Numbers, Rajchman Measures, and Small Cantor Sets,
\emph{Proceedings of the American Mathematical Society} Vol. 128, No. 9 (Sep., 2000), pp. 2637-2640

\bibitem[Bou]{NBour} Nicolas Bourbaki, \emph{Groupe et algèbre de Lie}, Chapitre 2 et 3, Springer, 2006

\bibitem[CK]{ChaiKon} Jon Chaika and David Constantine, Quantitative shrinking target properties for rotations and interval exchanges. \emph{Isr. J. of Math.} 230, 275–334 (2019). https://doi.org/10.1007/s11856-018-1824-8

\bibitem[CK2]{Chern} Nicolai Chernov and Dimitri Kleinbock, 
Dynamical Borel-Cantelli lemmas for Gibbs measures. \emph{Isr. J. Math.} 122 (2001) 1--27.

\bibitem[Del]{Adelz} Antoine Delzant, 
Groupes de Lie compacts et tores maximaux,
\emph{Séminaire Henri Cartan}, tome 12, no 1 (1959-1960), exp. no 1, p. 1-14

\bibitem[DF]{DolgoFay} Dmitri Dolgopyat and Bassam Fayad, Deviation of ergodic sums for toral translation II. boxes, \emph{Publ.math.IHES} 132, 293–352 (2020). https://doi.org/10.1007/s10240-020-00120-2

\bibitem[EPS]{Schmel} Fredrik Ekström, Tomas Persson, Jörg Schmeling, On the Fourier dimension and a modification, \emph{Journal of Fractal Geometry}, 2(3), 309-337. https://doi.org/10.4171/JFG/23

\bibitem[Kah]{Kaha}
Jean-Pierre Kahane, Sur la distribution de certaines séries aléatoires, Colloque de théorie des nombres (Bordeaux, 1969), \emph{Mémoires de la Société Mathématique de France}, no. 25 (1971), pp. 119-122. doi : 10.24033/msmf.42

\bibitem[Kau]{Kauf} R. Kaufman, On the theorem of Jarnik and Besicovitch, \emph{Acta Arith.} 39 (1981), 265-267.

\bibitem[Kes]{kesten} Harry Kesten, Uniform distribution mod 1,
\emph{Annals of mathematics} 71 (1960) 445--471

\bibitem[Kur]{Kurz} Jaroslav Kurzweil, On the metric theory of inhomogeneous diophantine approximations, \emph{Studia Mathematica} 15.1 (1955): 84-112

\bibitem[Parth]{Parth} Selection Theorems and their Applications ; Author. T. Parthasarathy ; Publisher. Springer ; Subject. Mathematics ; Year. 1972 ; Vol. 263.

\bibitem[Sol]{BoSol} Boris Solomyak, Fourier decay for self-similar measures,   \emph{Proceedings of the American Mathematical Society} (2021) 149(08):1

\bibitem[Spr]{Sprdz} V.G. Sprindzuk,  \emph{Metric theory of diophantine approximations} / Vladimir G. Sprindzuk ; translated and edited by Richard A. Silverman  V. H. Winston ; Wiley Washington, D.C. : New York  1979

\bibitem[Tho]{DaTho} Damien Thomine,
Keplerian shear in ergodic theory,
\emph{Annales Henri Lebesgue}, Volume 3 (2020), pp. 649-676.

\bibitem[Tse]{Tsen} Jimmy Tseng, On circle rotations and the shrinking target properties. \emph{Discrete and Continuous Dynamical Systems}, 2008, 20(4): 1111-1122. doi: 10.3934/dcds.2008.20.1111

\bibitem[Xin]{VicXing} Victoria Xing, Dynamical Borel–Cantelli Lemmas and Applications, 8th april 2020, LUTFMA-3402-2020

\bibitem[Swo]{Swor} Maciej Zworski, \emph{Semiclassical Analysis}. Graduate Studies in Mathematics 138. Amer. Math. Soc., Providence, RI, 2012
\end{thebibliography}

\end{document}